\documentclass[a4paper,11pt]{amsart}

\usepackage{amssymb}

\usepackage{wasysym}
\newcommand{\circt}{\mathbin{\ooalign{$\ocircle$\cr\hidewidth\raise-.15ex\hbox{$\scriptstyle\top\mkern2.1mu$}\cr}}} 
\newcommand{\smCirct}{\mathbin{\ooalign{$\scriptstyle\ocircle$\cr\hidewidth\raise-.12ex\hbox{$\scriptscriptstyle\top\mkern1mu$}\cr}}}  

\usepackage{hyperref}
\usepackage[initials,nobysame,alphabetic,msc-links]{amsrefs}
\DefineSimpleKey{bib}{how}

\renewcommand{\PrintDOI}[1]{%
    \href{http://dx.doi.org/#1}{{\tt DOI:#1}}%
}
\renewcommand{\eprint}[1]{#1}
\BibSpec{book}{%
    +{}  {\PrintPrimary}                {transition}
    +{.} { \PrintDate}                  {date}
    +{.} { \textit}                     {title}
    +{.} { }                            {part}
    +{:} { \textit}                     {subtitle}
    +{,} { \PrintEdition}               {edition}
    +{}  { \PrintEditorsB}              {editor}
    +{,} { \PrintTranslatorsC}          {translator}
    +{,} { \PrintContributions}         {contribution}
    +{,} { }                            {series}
    +{,} { \voltext}                    {volume}
    +{,} { }                            {publisher}
    +{,} { }                            {organization}
    +{,} { }                            {address}
    +{,} { }                            {status}
    +{,} { \PrintDOI}                   {doi}
    +{,} { \PrintISBNs}                 {isbn}
    +{}  { \parenthesize}               {language}
    +{}  { \PrintTranslation}           {translation}
    +{;} { \PrintReprint}               {reprint}
    +{.} { }                            {note}
    +{.} {}                             {transition}
    +{}  {\SentenceSpace \PrintReviews} {review}
}
\BibSpec{article}{%
    +{}  {\PrintAuthors}                {author}
    +{,} { \textit}                     {title}
    +{.} { }                            {part}
    +{:} { \textit}                     {subtitle}
    +{,} { \PrintContributions}         {contribution}
    +{.} { \PrintPartials}              {partial}
    +{,} { }                            {journal}
    +{}  { \textbf}                     {volume}
    +{}  { \PrintDatePV}                {date}
    +{,} { \issuetext}                  {number}
    +{,} { \eprintpages}                {pages}
    +{,} { }                            {status}
    +{,} { \PrintDOI}                   {doi}
    +{,} { \eprint}        {eprint}
    +{}  { \parenthesize}               {language}
    +{}  { \PrintTranslation}           {translation}
    +{;} { \PrintReprint}               {reprint}
    +{.} { }                            {note}
    +{.} {}                             {transition}
    +{}  {\SentenceSpace \PrintReviews} {review}
}
\BibSpec{collection.article}{%
    +{}  {\PrintAuthors}                {author}
    +{,} { \textit}                     {title}
    +{.} { }                            {part}
    +{:} { \textit}                     {subtitle}
    +{,} { \PrintContributions}         {contribution}
    +{,} { \PrintConference}            {conference}
    +{}  {\PrintBook}                   {book}
    +{,} { }                            {booktitle}
    +{,} { \PrintDateB}                 {date}
    +{,} { pp.~}                        {pages}
    +{,} { }                            {publisher}
    +{,} { }                            {organization}
    +{,} { }                            {address}
    +{,} { }                            {status}
    +{,} { \PrintDOI}                   {doi}
    +{,} { \eprint}        {eprint}
    +{}  { \parenthesize}               {language}
    +{}  { \PrintTranslation}           {translation}
    +{;} { \PrintReprint}               {reprint}
    +{.} { }                            {note}
    +{.} {}                             {transition}
    +{}  {\SentenceSpace \PrintReviews} {review}
}
\BibSpec{misc}{%
  +{}{\PrintAuthors}  {author}
  +{,}{ \textit}      {title}
  +{.}{ }             {how}
  +{}{ \parenthesize} {date}
  +{,} { available at \eprint}        {eprint}
  +{,}{ available at \url}{url}
  +{,}{ }             {note}
  +{.}{}              {transition}
}

\newcommand\Dhat{\hat\Delta}
\newcommand\sgn{\operatorname{sgn}}
\newcommand{\T}{\mathbb{T}}
\newcommand{\Tor}{\operatorname{Tor}}
\newcommand\bp{\begin{proof}}
\newcommand\ep{\end{proof}}
\newcommand\enu[1]{\smallskip\newline\makebox[5mm][l]{\rm(#1)}}
\newcommand{\SL}{\mathrm{SL}}
\newcommand{\SU}{\mathrm{SU}}

\newcommand{\Spin}{\mathrm{Spin}}
\newcommand{\U}{\mathrm{U}}

\newcommand{\liealg}[1]{\mathfrak{#1}}
\newcommand{\nat}{\textrm{nat}}
\newcommand{\pairing}[2]{\left \langle #1, #2 \right \rangle}
\newcommand\A{\mathcal A}
\newcommand{\C}{\mathbb{C}}

\newcommand{\Z}{\mathbb{Z}}
\newcommand{\N}{\mathbb{N}}
\newcommand{\G}{\mathbb{G}}
\newcommand\HH{\mathbb H}
\newcommand{\KK}{\mathrm{KK}}
\newcommand{\absv}[1]{\left \lvert #1 \right \rvert}

\newcommand{\ensm}[1]{\left \{ #1 \right \}}

\newcommand{\Hsp}{\mathcal{H}}
\newcommand{\Hilbf}{\mathrm{Hilb}_\textrm{f}}
\newcommand{\KZ}{\textrm{KZ}}
\newcommand{\ad}{\mathrm{ad}}
\newcommand{\Ad}{\mathrm{Ad}}
\newcommand{\Uni}{\mathcal{U}}
\newcommand{\EUni}{\tilde{\mathcal{U}}}
\newcommand\CC{\mathcal C}
\newcommand\RR{\mathcal R}

\DeclareMathOperator{\ract}{\lhd}
\DeclareMathOperator{\lact}{\rhd}

\DeclareMathOperator{\Ind}{Ind}
\DeclareMathOperator{\Prim}{Prim}
\DeclareMathOperator{\Stab}{Stab}
\DeclareMathOperator{\Rep}{Rep}
\DeclareMathOperator{\Irrep}{Irrep}
\DeclareMathOperator{\Map}{Map}
\DeclareMathOperator{\Hom}{Hom}

\DeclareMathOperator{\qdet}{qdet}

\numberwithin{equation}{section}

\theoremstyle{plain}
\newtheorem{thm}{Theorem}[section]
\newtheorem{prop}[thm]{Proposition}
\newtheorem{cor}[thm]{Corollary}
\newtheorem{lem}[thm]{Lemma}

\theoremstyle{definition}
\newtheorem{defn}[thm]{Definition}

\theoremstyle{remark}
\newtheorem{remk}[thm]{Remark}

\begin{document}

\title[twisting quantum groups]{Twisting the $q$-deformations of compact semisimple Lie groups}

\date{May 29, 2013; minor changes July 6, 2013}

\author[S. Neshveyev]{Sergey Neshveyev}

\email{sergeyn@math.uio.no}

\address{Department of Mathematics, University of Oslo,
P.O. Box 1053 Blindern, NO-0316 Oslo, Norway}

\thanks{The research leading to these results has received funding from the European Research Council
under the European Union's Seventh Framework Programme (FP/2007-2013) / ERC Grant Agreement no. 307663
}

\author[M. Yamashita]{Makoto Yamashita}

\email{yamashita.makoto@ocha.ac.jp}

\address{Department of Mathematical Sciences, University of Copenhagen,
Universitetsparken 5, 2100-K{\o}benhavn-\O, Denmark
(on leave from Ochanomizu University)}

\thanks{Supported by the Danish National Research Foundation through the Centre
for Symmetry and Deformation (DNRF92), and by JSPS KAKENHI Grant Number 25800058}

\begin{abstract}
Given a compact semisimple Lie group $G$ of rank $r$, and a parameter $q>0$, we can define new associativity morphisms in $\Rep(G_q)$ using a $3$-cocycle $\Phi$ on the dual of the center of~$G$, thus getting a new tensor category $\Rep(G_q)^\Phi$. For  a class of cocycles $\Phi$ we construct compact quantum groups~$G^\tau_q$ with representation categories $\Rep(G_q)^\Phi$. The construction depends on the choice of an $r$-tuple $\tau$ of elements in the center of $G$. In the simplest case of $G=\SU(2)$ and $\tau=-1$, our construction produces Woronowicz's quantum group $\SU_{-q}(2)$ out of $\SU_q(2)$. More generally, for $G=\SU(n)$, we get quantum group realizations of the Kazhdan--Wenzl categories.
\end{abstract}

\maketitle

\section*{Introduction}

A known problem in the theory of quantum groups is classification of quantum groups with fusion rules of a given Lie group $G$, see e.g.~\citelist{\cite{MR943923} \cite{MR1266253} \cite{MR1378260} \cite{MR1673475} \cite{MR2023750} \cite{MR2106933} \cite{mrozinski}}. Although this problem has been completely solved in a few cases, most notably for $G=\SL(2,\C)$ \citelist{\cite{MR1378260} \cite{MR2023750}}, as the rank of $G$ grows the situation quickly becomes complicated. Already for $G=\SL(3,\C)$, even when requiring the dimensions of the representations to remain classical, one gets a large list of quantum groups that is not easy to grasp \citelist{\cite{MR1673475} \cite{MR2106933}}. A categorical version of the same problem turns out to be more manageable. Namely, the problem is to classify semisimple rigid monoidal $\C$-linear categories with fusions rules of $G$. As was shown by Kazhdan and Wenzl~\cite{MR1237835}, for $G=\SL(n,\C)$ such categories $\CC$ are parametrized by pairs $(q_\CC,\tau_\CC)$ of nonzero complex numbers, defined up to replacing $(q_\CC,\tau_\CC)$ by $(q_\CC^{-1},\tau_\CC^{-1})$, such that $q_\CC^{n(n-1)/2}=\tau_\CC^n$ and $q_\CC$ is not a nontrivial root of unity.\footnote{This is not how the result is formulated in~\cite{MR1237835}. There is a known mistake in~\cite{MR1237835}*{Proposition~5.1}, see \cite{MR2825504}*{Section~7} for a discussion.} Concretely, these are twisted representation categories $\CC=\Rep(\SL_q(n))^\zeta$, where $q$ is not a nontrivial root of unity and $\zeta$ is a root of unity of order $n$; the corresponding parameters are $q_\CC=q^2$ and $\tau_\CC=\zeta^{-1}q^{n-1}$. The twists are defined by choosing a $\T$-valued $3$-cocycle on the dual of the center of $\SL(n,\C)$ and by using this cocycle to define new associativity morphisms in $\Rep(\SL_q(n))$. The third cohomology group of the dual of the center is cyclic of order $n$, and this explains the parametrization of twists of $\Rep(\SL_q(n))$ by roots of unity. A partial extension of the result of Kazhdan and Wenzl to types~$\mathrm{BCD}$ was obtained by Tuba and Wenzl~\cite{MR2132671}.

Although two problems are clearly related, a solution of the latter does not immediately say much about the former. The present work is motivated by the natural question whether there exist quantum groups with representation categories $\Rep(\SL_q(n))^\zeta$ for all $\zeta$ such that $\zeta^n=1$. Equivalently, do the categories $\Rep(\SL_q(n))^\zeta$ always admit fiber functors? For $n=2$ there is essentially nothing to solve, since for $q\ne1$ the category $\Rep(\SL_q(2))^{-1}$ is equivalent to $\Rep(\SL_{-q}(2))$. For $q=1$ the answer is also known: the quantum group $\SU_{-1}(2)$ defined by Woronowicz (which has nothing to do with the quantized universal enveloping algebra~$\Uni_{q}(\liealg{sl}_2)$ at~$q=-1$) has representation category $\Rep(\SL(2,\C))^{-1}$. For~$n\ge2$, quantum groups with fusion rules of $\SL(n,\C)$ have been studied by many authors, see e.g.~\cite{MR1809304} and the references therein. Usually, one starts by finding a solution of the quantum Yang--Baxter equation satisfying certain conditions, and from this derives a presentation of the algebra of functions on the quantum group~\cite{MR1015339}. This approach cannot work in our case, since the category $\Rep(\SL_q(n))^\zeta$ does not have a braiding unless $\zeta^2 = 1$.

The approach we take works, to some extent, for any compact semisimple simply connected Lie group $G$. Assume that $\Phi$ is a $\T$-valued $3$-cocycle on the dual of the center of $G$. To construct a fiber functor $\varphi$ from the category $\Rep(G_q)^\Phi$ with associativity morphisms defined by $\Phi$, such that $\dim \varphi(U)=\dim U$, is the same as to find an invertible element $F$ in a completion $\Uni(G_q\times G_q)$ of $\Uni_q(\liealg{g})\otimes \Uni_q(\liealg{g})$ satisfying
$$
\Phi = (\iota \otimes \hat{\Delta}_q)(F^{-1}) (1 \otimes F^{-1}) (F \otimes 1) (\hat{\Delta}_q \otimes \iota)(F).
$$
Then, using the twist (or a pseudo-$2$-cocycle in the terminology of~\cite{MR1395206}) $F$, we can define a new comultiplication on $\Uni(G_q)$, thus getting a new quantum group with representation category $\Rep(G_q)^\Phi$.

Our starting point is the simple remark that to solve the above cohomological equation we do not have to go all the way to $G_q$, it might suffice to pass from the center $Z(G)$ to a (quantum) subgroup of $G_q$, for example, to the maximal torus $T$. For simple $G$ this is indeed enough: any $3$-cocycle on $\widehat{Z(G)}$ becomes a coboundary when lifted to the dual $P=\hat T$ of $T$. The reason is that, for simple $G$, the center is contained in a torus of dimension at most $2$. However, a $2$-cochain $f$ on $P$ such that $\partial f=\Phi$ is unique only up to a $2$-cocycle on $P$. Already for trivial $\Phi$ this leads to deformations of~$G_q$ by $2$-cocycles on $P$ that are not very well studied~\citelist{\cite{MR1127037} \cite{MR1116413}}, with associated C$^*$-algebras of functions (for $q>0$) that are typically not of type I.

Our next observation is that, for arbitrary~$G$, if $\Phi$ lifts to a coboundary on $P$, then the cochain $f$ can be chosen to be of a particular form. This leads to a very special class of quantum groups $G^\tau_q$, whose construction depends on the choice of elements $\tau_1,\dots,\tau_r\in Z(G)$, where $r$ is the rank of $G$. We show that the quantum groups~$G^\tau_q$ are as close to~$G_q$ as one could hope. For example, they can be defined in terms of finite central extensions of~$\Uni_q(\liealg{g})$.

Since we are, first of all, interested in compact quantum groups in the sense of Woronowicz, we will concentrate on the case $q>0$, when the categories $\Rep(G_q)^\Phi$ have a C$^*$-structure and, correspondingly, $G^\tau_q$ become compact quantum groups. We then show that the C$^*$-algebras $C(G^\tau_q)$ are $\KK$-isomorphic to $C(G)$, they are of type I, and their primitive spectra are only slightly more complicated than that of $C(G_q)$. For $G=\SU(n)$ we also find explicit generators and relations of the algebras $\C[\SU^\tau_q(n)]$ of regular functions on $\SU^\tau_q(n)$.

To summarize, our construction produces quantum groups with nice properties and with representation category $\Rep(G_q)^\Phi$ for any $3$-cocycle $\Phi$ on $\widehat{Z(G)}$ that lifts to a coboundary on $\hat T$. This covers the cases when $G$ is simple, but in the general semisimple case there exist cocycles that do not have this property. For such cocycles the existence of fiber functors for $\Rep(G_q)^\Phi$ remains an open problem.

\medskip

\paragraph{\bf Acknowledgment} We would like to thank Kenny De Commer for stimulating discussions and valuable comments.

\bigskip

\section{Preliminaries}

\subsection{Compact quantum groups}
\label{sec:comp-quant-groups}

A \textit{compact quantum group} $\G$ is given by a unital C$^*$-algebra $C(\G)$ together with a coassociative unital $*$-homomorphism $\Delta\colon C(\G) \rightarrow C(\G) \otimes C(\G)$ satisfying the cancellation condition
\[
[\Delta(C(\G)) (C(\G) \otimes 1)] = C(\G) \otimes C(\G) = [\Delta(C(\G)) (1 \otimes C(\G))],
\]
where brackets denote the closed linear span. Here we only introduce the relevant terminology and summarize the essential results, see e.g.~\cite{neshveyev-tuset-book} for details.

A theorem of Woronowicz gives a distinguished state $h$, the Haar state, which is an analogue of the normalized Haar measure over compact groups. Denote by $C_r(\G)$ the quotient of $C(\G)$ by the kernel of the GNS-representation defined by $h$.  We will be interested in the case where $h$ is faithful, so that $C_r(\G)=C(\G)$.  This condition is automatically satisfied for coamenable compact quantum groups.  The quantum groups studied in this paper will be coamenable thanks to Banica's theorem~\citelist{\cite{MR1679171}*{Proposition~6.1} \cite{neshveyev-tuset-book}*{Theorem~2.7.14}}.

A finite dimensional unitary representation of $\G$ is given by a unitary element $U \in B(\Hsp_U) \otimes C(\G)$ satisfying the condition $U_{1 3} U_{2 3} = (\iota \otimes \Delta)(U)$. The tensor product of two representations is defined by $U\circt V=U_{13}V_{23}$. The category $\Rep(\G)$ of finite dimensional unitary representations of $\G$ has the structure of a rigid C$^*$-tensor category with a unitary fiber functor (`forgetful functor') $U \mapsto \Hsp_U$ to the category $\Hilbf$ of finite dimensional Hilbert spaces. Woronowicz's Tannaka--Krein duality theorem states that the reduced quantum group $(C_r(\G), \Delta)$ can be axiomatized in terms of $\Rep(\G)$ and the fiber functor.

We denote by $\C[\G]\subset C(\G)$ the Hopf $*$-algebra of matrix coefficients of finite dimensional representations of $\G$. Denote by $\Uni(\G)$ the dual $*$-algebra of $\C[\G]$, so $\Uni(\G) = \prod_{U \in \Irrep(\G)} B(\Hsp_U)$. It can be considered from many different angles: as the algebra of functions on the dual discrete quantum group $\hat\G$, as the algebra of endomorphisms of the forgetful functor, as the multiplier algebra of the convolution algebra $\widehat{\C[\G]}$ of $\G$. We also write $\Uni(\G^n)$ for $n \ge 2$ to denote the `tensor product' multipliers, such as
\[
\Uni(\G^2) = \prod_{U, V \in \Irrep(\G)} B(\Hsp_U) \otimes B(\Hsp_V).
\]
By duality, the multiplication map $m\colon\C[\G]\otimes\C[\G]\to\C[\G]$ defines a `coproduct' $\Dhat\colon\Uni(\G)\to\Uni(\G^2)$.

\subsection{Twisting of quantum groups}
\label{sec:twist-group-cocycl}

Let $\G$ be a compact quantum group, and $\Phi$ be an invariant unitary $3$-cocycle over the discrete dual of $\G$~\cite{neshveyev-tuset-book}*{Chapter~3}.  Thus, $\Phi$ is a unitary element in~$\Uni(\G^3)$ satisfying the cocycle condition
\begin{equation}\label{eq:3-cocycle}
( 1 \otimes \Phi ) (\iota \otimes \hat{\Delta} \otimes \iota)(\Phi) ( \Phi \otimes 1 ) = (\iota \otimes \iota \otimes \hat{\Delta})(\Phi) (\hat{\Delta} \otimes \iota \otimes \iota)(\Phi)
\end{equation}
and the invariance condition $[\Phi, (\hat{\Delta} \otimes \iota) \hat{\Delta} (x)] = 0$ for $x \in \Uni(\G)$.

Then, the representation category $\Rep(\G)$ can be twisted into a new C$^*$-tensor category $\Rep(\G)^\Phi$, by using the action by $\Phi$ on $\Hsp_U \otimes \Hsp_V \otimes \Hsp_W$ as the new associativity morphism $(U \circt V) \circt W \rightarrow U \circt (V \circt W)$ for $U, V, W \in \Rep(\G)$. The category $\Rep(\G)^\Phi$ can be considered as the module category of the discrete quasi-bialgebra $(\widehat{\C[\G]}, \hat{\Delta}, \Phi)$~\cite{MR1047964}.

Suppose the category $\Rep(\G)^\Phi$ is rigid. This is equivalent to the condition that the central element
$$
\Phi_1\hat S(\Phi_2)\Phi_3=m(m\otimes\iota)(\iota\otimes\hat S\otimes\iota)(\Phi)
$$
in $\Uni(\G)$ is invertible. Suppose also that there exists a  unitary $F \in \Uni(\G^2)$ such that
\begin{equation}\label{eq:Phi-cobdry-of-F}
\Phi = (\iota \otimes \hat{\Delta})(F^*) (1 \otimes F^*) (F \otimes 1) (\hat{\Delta} \otimes \iota)(F).
\end{equation}
Then the discrete quantum group $\Uni(\G)$ can be deformed into another one, with the new coproduct
$
\hat{\Delta}_F(x) = F \hat{\Delta}(x) F^*.
$
By duality, the function algebra $\C[\G]$ can be endowed with the new product $$x \cdot_F y = m(F^* \lact (x \otimes y) \ract F).$$  Here, $\lact$ and $\ract$ are the natural actions of $\Uni(\G)$ on $\C[\G]$ given by $X \lact a = \pairing{X}{a_{[2]}} a_{[1]}$ and $a \ract X = \pairing{X}{a_{[1]}} a_{[2]}$.  We denote the corresponding compact quantum group by $\G_F$. Note that in general there is no simple formula for the involution on $\C[\G_F]$, compare with \cite{neshveyev-tuset-book}*{Example~2.3.9}.

We have a unitary monoidal equivalence of the C$^*$-tensor categories $\Rep(\G)^\Phi$ and $\Rep(\G_F)$.  The tensor functor $\varphi\colon \Rep(\G)^\Phi\to \Rep(\G_F)$ is given by the identity map on objects and morphisms, but with the nontrivial tensor transformation $\varphi(U)\circt\varphi(V)\to\varphi(U\circt V)$ defined by
\[
\Hsp_U \otimes \Hsp_V \rightarrow \Hsp_U \otimes \Hsp_V, \quad \xi \otimes \eta \mapsto F^* (\xi \otimes \eta).
\]
In terms of fiber functors, $F$ gives a tensor functor $\Rep(\G)^\Phi \rightarrow \Hilbf$ which is the same as that of $\Rep(\G)$ on objects and morphisms, but with the modified tensor transformation $\Hsp_{U} \otimes \Hsp_V\to \Hsp_{U \smCirct V}$ given by $\xi \otimes \eta \mapsto F^*(\xi \otimes \eta)$.

Examples of invariant $3$-cocycles can be obtained as follows. Assume $\HH$ is a closed central subgroup of~$\G$, so $\HH$ is a compact abelian group and we are given a surjective homomorphism $\pi\colon\C[\G]\to\C[\HH]$ of Hopf $*$-algebras such that the image of $\Uni(\HH)$ under the dual homomorphism $\Uni(\HH)\to\Uni(\G)$ is a central subalgebra of $\Uni(\G)$, or equivalently, for any irreducible unitary representation~$U$ of~$\G$ the element $(\iota\otimes\pi)(U)$ has the form $1\otimes\chi_U$ for a character $\chi_U$ of $\HH$. Unitary $3$-cocycles in $\Uni(\HH^3)$ are nothing else than $\T$-valued 3-cocycles on the Pontryagin dual $\hat{\HH}$. Any such cocycle defines an invariant cocycle $\Phi$ in $\Uni(\G^3)$; when $\G$ is itself compact abelian, this is just the usual pullback homomorphism $Z^3(\hat\HH; \T) \rightarrow Z^3(\hat{\G}; \T)$. Explicitly, the action of $\Phi$ on $\Hsp_U \otimes \Hsp_V\otimes\Hsp_W$ is by multiplication by $\Phi(\chi_U,\chi_V,\chi_W)$. For such cocycles $\Phi$ the C$^*$-tensor category $\Rep(\G)^\Phi$ is always rigid.

\subsection{Quantized universal enveloping algebra}
\label{sec:quant-univ-envel}

Throughout the whole paper $G$ denotes a semi\-simple simply connected compact Lie group, and $\liealg{g}$ denotes its complexified Lie algebra.  We fix a maximal torus $T$ in $G$, and denote the corresponding Cartan subalgebra by $\liealg{h}$.  The root lattice is denoted by $Q$, and the weight lattice by $P$.  We fix a choice of positive roots, and denote the corresponding positive simple roots by $\ensm{\alpha_1, \ldots, \alpha_r}$. We also fix an $\ad$-invariant symmetric form on~$\liealg{g}$ such that it is negative definite on the real Lie algebra of $G$. If $G$ is simple, we assume that this form is standardly normalized, meaning that $(\alpha,\alpha)=2$ for every short root $\alpha$. The Cartan matrix is denoted by $(a_{i j})_{1 \le i, j \le r}$, and the Weyl group is denoted by $W$.  The center $Z(G)$ of $G$ is contained in $T$ and can be identified with the dual of $P/Q$.

The variable $q$ ranges over the strictly positive real numbers, although many results remain true for all $q\ne0$ that are not nontrivial roots of unity. For $q \neq 1$, the \textit{quantized universal enveloping algebra} $\Uni_q(\liealg{g})$ is the universal algebra over $\C$ generated by the elements $E_i$, $F_i$, and $K_i^{\pm 1}$ for $1 \le i \le r$ satisfying the relations
\begin{gather*}
[K_i, K_j] = 0, \quad K_i E_j K_i^{-1} = q_i^{a_{i j}} E_j, \quad K_i F_j K_i^{-1} = q_i^{-a_{i j}} F_j,\\
[E_i, F_j] = \delta_{i j} \frac{K_i - K_i^{-1}}{q_i - q_i^{-1}},\\
\sum_{k = 0}^{1-a_{i j}} (-1)^k
\left [ \begin{array}{c}
  1 - a_{i j}\\
  k
\end{array} \right ]_{q_i}
E_i^k E_j E_i^{1 - a_{i j} - k} = 0,\\
\sum_{k = 0}^{1-a_{i j}} (-1)^k
\left [ \begin{array}{c}
  1 - a_{i j}\\
  k
\end{array} \right ]_{q_i}
F_i^k F_j F_i^{1 - a_{i j} - k} = 0,
\end{gather*}
where $q_i = q^{(\alpha_i, \alpha_i) / 2}$.  It has the structure of a Hopf $*$-algebra defined by the operations
\begin{gather*}
\hat{\Delta}_q(E_i) = E_i \otimes 1 + K_i \otimes E_i,\quad
\hat{\Delta}_q(F_i) = F_i \otimes K_i^{-1} + 1 \otimes F_i, \quad
\hat{\Delta}_q(K_i) = K_i \otimes K_i,\\
\hat{S}_q(E_i) = - K_i^{-1} E_i,\quad
\hat{S}_q(F_i) = - F_i K_i^{-1},\quad
\hat{S}_q(K_i) = K_i^{-1},\\
\hat{\epsilon}_q(E_i) = \hat{\epsilon}_q(F_i) = 0, \quad \hat{\epsilon}_q(K_i) = 1,\\
E_i^* = F_i K_i,\quad
F_i^* = K_i^{-1} E_i, \quad
K_i^* = K_i.
\end{gather*}

A representation $(\pi, V)$ of $\Uni_q(\liealg{g})$ is said to be \textit{admissible} when $V$ admits a decomposition $\oplus_{\chi \in P} V_\chi$ such that $\pi(K_i)|_{V_\chi}$ is equal to the scalar $q^{(\alpha_i, \chi)}$. The category of finite dimensional admissible $*$-representations of $\Uni_q(\liealg{g})$ is a C$^*$-tensor category with the forgetful functor.  We denote the associated compact quantum group by $G_q$.  There is a natural inclusion of $T$ into $\Uni(G_q)$.   Then the set  $Z(G_q)$ of group-like central elements in $\Uni(G_q)$ coincides with $Z(G)$.  The class of representations of~$G_q$ on which $Z(G)$ acts trivially corresponds to a quotient quantum group denoted by~$G_q/Z(G)$.

\bigskip

\section{Twisted \texorpdfstring{$q$}{q}-deformations}
\label{sec:twist-repn}

\subsection{Extension of the QUE-algebra}
\label{sec:twisting-que}

For $q>0$, we let $\EUni_q(\liealg{g})$ denote the universal $*$-algebra generated by $\Uni_q(\liealg{g})$ and unitary central elements $C_1, \ldots, C_r$. It is not difficult to check that for $q\ne1$ the following formulas define a Hopf $*$-algebra structure on $\EUni_q(\liealg{g})$:
\begin{align*}
\hat{\Delta}(E_i) &= E_i \otimes C_i + K_i \otimes E_i,&
\hat{\Delta}(K_i) &= K_i \otimes K_i,&
\hat{\Delta}(C_i) &= C_i \otimes C_i.
\end{align*}
Similarly, for $q = 1$, we define
\begin{align*}
\hat{\Delta}(E_i) &= E_i \otimes C_i + 1 \otimes E_i,&
\hat{\Delta}(H_i) &= H_i \otimes 1 + 1 \otimes H_i,&
\hat{\Delta}(C_i) &= C_i \otimes C_i.
\end{align*}

There is a Hopf $*$-algebra homomorphism from $\EUni_q(\liealg{g})$ onto $\Uni_q(\liealg{g})$, defined by $C_i \mapsto 1$ and by the identity map on the copy of $\Uni_q(\liealg{g})$.  There is also a Hopf $*$-algebra homomorphism onto $\C[(C_i)_{i=1}^r]$, given by $E_i \mapsto 0$, $F_i \mapsto 0$, $K_i \mapsto 1$, and by the identity map on the $C_i$'s.  We regard representations of~$\Uni_q(\liealg{g})$ and of $\C[(C_i)_{i = 1}^r]$ as the ones of $\EUni_q(\liealg{g})$ via these homomorphisms.

\begin{remk}
The Hopf algebra $\EUni_q(\liealg{g})$ is closely related to the Drinfeld double $\mathcal{D}(\Uni_q(\liealg{b}_+))$ of $\Uni_q(\liealg{b}_+) = \langle E_i, K_i \mid 1 \le i \le r\rangle$. Namely, put
 \[
   X_i^+ = E_i C_i^{-1}, \quad K^+_i = K_i C_i^{-1}, \quad X_i^- = F_i, \quad K^-_i = K_i C_i.
 \]
Then we see that the elements $X_i^+$ and $K^+_i$ generate a copy of $\Uni_q(\liealg{b}_+)$, while the $X_i^-$ and $K_i^-$ generate a copy of $\Uni_q(\liealg{b}_-)$, and taking together these subalgebras give a copy of $\mathcal{D}(\Uni_q(\liealg{b}_+))$ in $\EUni_q(\liealg{g})$.The homomorphism $\EUni_q(\liealg{g}) \rightarrow \Uni_q(\liealg{g})$ is an extension of the standard projection $\mathcal{D}(\Uni_q(\liealg{b}_+)) \rightarrow \Uni_q(\liealg{g})$.  If we add square roots of $K^\pm_i$ to $\mathcal{D}(\Uni_q(\liealg{b}_+))$, thus getting a Hopf algebra $\widetilde{\mathcal{D}(\Uni_q(\liealg{b}_+))}$, we can recover $\EUni_q(\liealg{g})$ by letting $C_i=(K^-_i)^{1/2}(K^+_i)^{-1/2}$. Therefore we have inclusions of Hopf algebras $\mathcal{D}(\Uni_q(\liealg{b}_+))\subset \EUni_q(\liealg{g})\subset \widetilde{\mathcal{D}(\Uni_q(\liealg{b}_+))}$.
\end{remk}

Let $\tau = (\tau_1, \ldots, \tau_r)$ be an $r$-tuple of elements in $Z(G)$. We say that a representation $(\pi,V)$ of $\EUni_q(\liealg{g})$ is {\em $\tau$-admissible} if its restriction to $\Uni_q(\liealg{g})$ is admissible and the elements $C_i$ act on the weight spaces~$V_\chi$ as scalars $\pairing{\tau_i}{\chi}$. The category of $\tau$-admissible representations is a rigid C$^*$-tensor category with forgetful functor.  Moreover, the $G_q/Z(G)$-representations are naturally included in the $\tau$-admissible representations as a C$^*$-tensor subcategory.

\begin{defn}
We let $G_q^\tau$ denote the compact quantum group realizing the category of finite dimensional $\tau$-admissible $*$-representations of $\EUni_q(\liealg{g})$ together with its canonical fiber functor.
\end{defn}

In other words, $\C[G^\tau_q]\subset\EUni_q(\liealg{g})^*$ is spanned by matrix coefficients of finite dimensional $\tau$-admissible representations, and the Hopf $*$-algebra structure on $\C[G^\tau_q]$ is defined by duality using that of $\EUni_q(\liealg{g})$.

Since every admissible representation of  $\Uni_q(\liealg{g})$ extends uniquely to a $\tau$-admissible representation of $\EUni_q(\liealg{g})$, and every $\tau$-admissible representation is obtained this way, we can identify the $*$-algebra $\Uni(G^\tau_q)$ with $\Uni(G_q)$. The image $\Uni^\tau_q(\liealg{g})$ of $\EUni_q(\liealg{g})$ in $\Uni(G^\tau_q)=\Uni(G_q)$ plays the role of a quantized universal enveloping algebra for $G^\tau_q$. As an algebra it is generated by $E_i$, $F_i$, $K_i^{\pm1}$ and $\tau_i$ (which is the image of $C_i$), but is endowed with a modified coproduct
\begin{align}\label{eq:twist-coprod-UGqtau}
\hat{\Delta}(E_i) &= E_i \otimes \tau_i + K_i \otimes E_i,&
\hat{\Delta}(K_i) &= K_i \otimes K_i,&
\hat{\Delta}(\tau_i) &= \tau_i \otimes \tau_i.
\end{align}
To put it differently,  as a $*$-algebra, $\Uni^\tau_q(\liealg{g})$ is the tensor product of $\Uni_q(\liealg{g})$ and the group algebra of the group $T_\tau \subset Z(G)$ generated by $\tau_1,\dots,\tau_r$, while the coproduct is defined by~\eqref{eq:twist-coprod-UGqtau}. As a quotient of $\EUni_q(\liealg{g})$, the Hopf $*$-algebra $\Uni^\tau_q(\liealg{g})$ is obtained by requiring that the unitaries $C_1,\dots,C_r$ satisfy the same relations as $\tau_1,\dots,\tau_r\in Z(G)$.

\subsection{Twisting and associator}
\label{sec:twisting-associator}

Given $\tau = (\tau_1, \ldots, \tau_r) \in Z(G)^r$, we obtain a $3$-cocycle on $\widehat{Z(G)}=P/Q$ as follows.

First, let $f(\lambda, \mu)$ be a $\T$-valued function on $P \times P$ satisfying
\begin{align}\label{eq:f-from-tau}
  f(\lambda, \mu + Q) &= f(\lambda, \mu),&
  f(\lambda + \alpha_i, \mu) &= \pairing{\tau_i}{\mu} f(\lambda, \mu).
\end{align}
These conditions imply that $f$ can be determined by its restriction to the image of a set-theoretic section $(P/Q)^2 \rightarrow P^2$. For example, if $\lambda_1, \ldots, \lambda_n$  is a system of representatives of $P/Q$, then we can put
\[
f\Big(\lambda_i + \sum_{j = 1}^r m_j \alpha_j, \mu\Big) = \prod_{j=1}^r \pairing{\tau_j}{\mu}^{m_j}
\]
for all $1 \le i \le n$ and $(m_1, \ldots, m_r) \in \Z^r$.

Using~\eqref{eq:f-from-tau}, the coboundary of $f$,
\[
(\partial f)(\lambda, \mu, \nu) = f(\mu, \nu) f(\lambda + \mu, \nu)^{-1} f(\lambda, \mu + \nu) f(\lambda, \mu)^{-1},
\]
is seen to be invariant under the translation by $Q$ in each variable.  Thus, $\partial f$ can be considered as a $3$-cochain on $P/Q$ with values in $\T$.  By construction, it is a cocycle. If $f'$ satisfies the same condition as $f$ above, the difference $f'f^{-1}$ is $Q^2$-invariant, that is, it defines a function on $(P/Q)^2$.  Thus, the cohomology class of $\partial f$ in $H^3(P/Q; \T)$ depends only on $\tau$. It also follows that the twisted coproduct $\hat{\Delta}_f(x) = f \hat{\Delta}_q(x) f^*$ does not depend on the choice of $f$.

Since $(\partial f)^*$ belongs to $\Uni(Z(G)^3)$, as we discussed in Section~\ref{sec:twist-group-cocycl}, it can be regarded as an invariant $3$-cocycle in $\Uni(G_q^3)$ which is denoted by $\Phi^\tau$. Similarly, $f$ can be considered as a unitary in $\Uni(G_q^2)$, and we have
$$
\Phi^\tau = (\iota \otimes \hat{\Delta}_q)(f^*) (1 \otimes f^*) (f \otimes 1) (\hat{\Delta}_q \otimes \iota)(f).
$$

\begin{prop}
The coproduct $\Dhat_f$ on $\Uni(G_q)$ coincides with the coproduct $\Dhat$ defined by \eqref{eq:twist-coprod-UGqtau}. \end{prop}

\begin{proof}
Since $f$ is contained in $\Uni(T^2)\subset\Uni(G_q^2)$, $\hat{\Delta}_f=\Dhat_q$ on the elements $K_i$.  For $E_i$, since the action of $E_i$ on an admissible module increases the weight of a vector by $\alpha_i$, identities  \eqref{eq:f-from-tau} imply that $f (K_i \otimes E_i) f^* = K_i \otimes E_i$ and $f (E_i \otimes 1) f^* = E_i \otimes \tau_i$.  Comparing these identities with~\eqref{eq:twist-coprod-UGqtau}, we obtain the assertion.
\end{proof}

\begin{cor}\label{cor:cat-equiv-gen-form}
The representation category of $G_q^\tau$ is unitarily monoidally equivalent to $\Rep(G_q)^{\Phi^\tau}$, the representation category of $G_q$ with associativity morphisms defined by $\Phi^\tau$.
\end{cor}

This result can also be interpreted as follows. Let $\Phi_{\KZ, q}\in\Uni(G^3)$ be the Drinfeld associator coming from the Knizhnik--Zamolodchikov equations associated with the parameter $\hbar = \log (q) / \pi i$. The representation category of $G_q$ is equivalent to that of $G$ with associativity morphisms defined by $\Phi_{\KZ, q}$.  The equivalence is given by a unitary Drinfeld twist $F_D \in \Uni(G^2)$ satisfying~\eqref{eq:Phi-cobdry-of-F} for $\Phi_{\KZ, q}$~\cite{neshveyev-tuset-book}*{Chapter~4}. It follows that $\Rep(G^\tau_q)$ is unitarily monoidally equivalent to the category $\Rep(G)$ with associativity morphisms defined by
$$
\Phi^\tau_{\KZ,q}=(\iota \otimes \hat{\Delta})(F_D^*) (1 \otimes F_D^*) \Phi^\tau(F_D \otimes 1) (\hat{\Delta} \otimes \iota)(F_D)=\Phi^\tau\Phi_{\KZ,q},
$$
where we now consider $\Phi^\tau$ as an element of $\Uni(G^3)$. Correspondingly, the unitary $F^\tau_D=fF_D\in\Uni(G^2)$ plays the role of a Drinfeld twist for $G^\tau_q$.

\begin{remk}\label{remk:isospec-deform-Dirac}
The construction of~\cite{MR2643923} can be carried out for $G_q^\tau$ to obtain a spectral triple over $\C[G_q^\tau]$ as an isospectral deformation of the spin Dirac operator on $G$.  Indeed, it is enough to verify the boundedness of $[1 \otimes (\iota \otimes \gamma)(t), (\pi \otimes \iota \otimes \widetilde{\ad})(\Phi_{\KZ,q}^\tau)]$ for any irreducible representation~$\pi$, where $t$ is the standard symmetric tensor $\sum_i x_i \otimes x_i$~\cite{MR2643923}*{Corollary~3.2}.  Since $(\pi \otimes \iota \otimes \widetilde{\ad})(\Phi^\tau) \in \C \otimes \Uni(Z(G)) \otimes \C$ commutes with $1 \otimes (\iota \otimes \gamma)(t)$, we can reduce the proof to the case of trivial $\tau$.
\end{remk}

A natural question is how large the class of cocycles of the form $\Phi^\tau$ is. These cocycles are analyzed in detail in Appendix. Using that analysis we point out the following.

\begin{prop}
A $\T$-valued $3$-cocycle $\Phi$ on $P/Q$ is cohomologous to $\Phi^\tau$ for some $\tau_1,\dots,\tau_r\in Z(G)$ if and only if $\Phi$ lifts to a coboundary on $P$. This is always the case if $P/Q$ can be generated by not more than two elements. For example, this is the case if $G$ is simple.
\end{prop}

\bp The first statement is proved in Corollary~\ref{ccmain}. It is also shown there that another equivalent condition on $\Phi$ is that it vanishes on $\bigwedge^3(P/Q)\subset H_3(P/Q;\Z)$. This condition is obviously satisfied if $P/Q$ can be generated by two elements. Finally, if $G$ is simple, then it is known that $P/Q$ is cyclic in all cases except for $G=\Spin(4n)$, in which case $P/Q\cong\Z/2\Z\oplus\Z/2\Z$.
\ep

Therefore for simple $G$ the quantum groups $G^\tau_q$ realize all possible associativity  morphisms on $\Rep(G_q)$ defined by $3$-cocycles on the dual of the center. In the semisimple case this is not true as soon as the center becomes slightly more complicated, namely, as soon as $\bigwedge^3(P/Q)\ne0$. We conjecture that in this case, if we take a cocycle $\Phi$ on $P/Q$ that does not lift to a coboundary on $P$, then there are no unitary fiber functors on $\Rep(G)^\Phi$, that is, there are no compact quantum groups with this representation category. Note that by Corollary~\ref{ccmain2} any such cocycle $\Phi$ is cohomologous to product of a cocycle $\Phi^\tau$ and a $3$-character on $P/Q$ that is nontrivial on $\bigwedge^3(P/Q)\subset(P/Q)^{\otimes3}$.

\subsection{Isomorphisms of twisted quantum groups}
\label{sec:isomorphisms}

Denote the cohomology class of the cocycle~$\Phi^\tau$ in $H^3(P/Q;\T)$ by $\Theta(\tau)$. This way we obtain a homomorphism
\[
\Theta\colon Z(G)^r \rightarrow H^3(P/Q; \T).
\]
Assume $\tau\in\ker\Theta$. Let $f$ be a function satisfying \eqref{eq:f-from-tau}. Then there exists a $2$-cochain $g\colon (P/Q)^2\to\T$ such that $\partial f=\partial g$, so that $fg^{-1}$ is a $2$-cocycle on $P$. Another choice of $f$ and $g$ would give us a cocycle that differs from $fg^{-1}$ by a $2$-cocycle on $P/Q$. Therefore taking the cohomology class of~$fg^{-1}$ we get a well-defined homomorphism
$$
\Upsilon\colon\ker\Theta\to H^2(P;\T)/H^2(P/Q;\T).
$$

\begin{prop}\label{piso}
Assume $\tau',\tau\in Z(G)^r$ are such that
$$
\tau'\tau^{-1}\in\ker\Theta \ \ \text{and}\ \ \tau'\tau^{-1}\in\ker\Upsilon.
$$
Then the quantum groups $G_q^{\tau'}$ and $G_q^\tau$ are isomorphic.
\end{prop}

\bp Denote by $\Dhat'$ and $\Dhat$ the coproducts on $\Uni(G_q)$ defined by $\tau'$ and $\tau$, see~\eqref{eq:twist-coprod-UGqtau}. Let $f'$ and $f$ be functions satisfying \eqref{eq:f-from-tau} for $\tau'$ and $\tau$, respectively, so that $\Dhat'=\Dhat_{f'}$ and $\Dhat=\Dhat_f$. The assumptions of the proposition mean that there exists a $2$-cochain $g$ on $P/Q$ such that $f'f^{-1}g$ is a coboundary on $P$. In other words, there exists a unitary $u\in\Uni(T^2)\subset\Uni(G_q^2)$ such that
$$
f'g=(u\otimes u)f\Dhat_q(u)^*.
$$
Then $\Ad u$ is an isomorphism of $(\Uni(G_q),\Dhat)$ onto $(\Uni(G_q),\Dhat')$, hence $G_q^\tau\cong G_q^{\tau'}$.
\ep

Apart from the isomorphisms given by this proposition, we have $G^\tau_q\cong G^{\tau^{-1}}_{q^{-1}}$. There also are isomorphisms induced by symmetries of the based root datum of $G$. Finally, for $q=1$ there can be additional isomorphisms defined by conjugation by elements in $\Uni(G)$ that lie in the normalizer of the maximal torus.

\bigskip

\section{Function algebras of twisted quantum groups}
\label{sec:funct-algebra-twist}

\subsection{Crossed product description}
\label{sec:comp-quant-group-realiz}

As before, assume $\tau=(\tau_1,\dots,\tau_r)\in Z(G)^r$. Recall that we denote by $T_\tau$ the subgroup of $Z(G)$ generated by the elements $\tau_1,\dots,\tau_r$.
There is a homomorphism $$\psi\colon \hat{T}_\tau\to T/Z(G)$$ defined as follows.  Given $\chi \in \hat{T}_\tau$, we define a character on the root lattice $Q$ by $\alpha_i \mapsto \chi(\tau_i)$. It can be extended to $P$, and we obtain an element $\tilde{\psi}(\chi) \in \hat{P} = T$.  The ambiguity of this extension is in $Q^\perp \cap T =Z(G)$.  Thus, the image $\psi(\chi)$ of $\tilde{\psi}(\chi)$ in $T/Z(G)$ is well-defined.

The homomorphism $\psi$ allows us to define an action of $\hat T_\tau$ by conjugation on $G_q$, that is, we have an action $\Ad\psi$ of $\hat T_\tau$ on $C(G_q)$ defined by
$$
(\Ad \psi(\chi))(a)={\pairing{\tilde\psi(\chi^{-1})}{a_{[1]}}} \pairing{\tilde\psi(\chi)}{a_{[3]}}a_{[2]};
$$
recall that the elements of $T$ define characters of $C(G_q)$, that is, they are group-like unitary elements in~$\Uni(G_q)$.

\begin{thm}\label{thm:fixed-pt-in-crossed-prod}
There is a canonical isomorphism
  \[
  C(G_q^\tau) \cong (C(G_q) \rtimes_{\Ad \psi} \hat{T}_\tau)^{T_\tau},
  \]
where the group $T_\tau$ acts on $C(G_q) \rtimes_{\Ad \psi} \hat{T}_\tau$ by right translations $\rho$ on $C(G_q)$ and by the dual action on~$C^*(\hat{T}_\tau)$.
\end{thm}

\begin{proof}
Let us first identify the compact quantum group $\tilde G_q^\tau$ defined by the category of finite dimensional representations of $\Uni_q^\tau(\liealg{g})$ such that their restrictions to $\Uni_q(\liealg{g})$ are admissible. Any such irreducible representation is tensor product of an irreducible admissible representation of $\Uni_q(\liealg{g})$ and a character of $T_\tau$; recall that these can be regarded as representations of $\Uni^\tau_q(\liealg{g})$. It follows that the Hopf $*$-algebra $\C[\tilde G^\tau_q]$ contains copies of $\C[G_q]$ and $C^*(\hat T_\tau)$, and as a space it is tensor product of these Hopf $*$-subalgebras. It remains to find relations between elements of $\C[G_q]$ and $C^*(\hat T_\tau)$ inside~$\C[\tilde G_q^\tau]$.

Let $(\pi, V)$ be a finite dimensional admissible representation of $\Uni_q(\liealg{g})$, and $\chi$ be a character of $T_\tau$.  Then, on the one hand, $\pi \otimes \chi$ is a representation on $V$ with $E_i$ acting by $\chi(\tau_i) \pi(E_i)$.  On the other hand, $\chi \otimes \pi$ is also a representation on the same space $V$ with $E_i$ acting by $\pi(E_i)$. From this we see that the operator $\pi(\tilde\psi(\chi))$, where we consider the standard extension of $\pi$ to $\Uni(G_q)$, intertwines $\chi\otimes\pi$ with $\pi\otimes\chi$. In other words, if $U_\pi\in B(V)\otimes\C[G_q]$ is the representation of $G_q$ defined by $\pi$, then in $B(V)\otimes\C[\tilde G^\tau_q]$ we have
$$
(\pi(\tilde\psi(\chi))\otimes u_\chi)U_\pi=U_\pi(\pi(\tilde\psi(\chi))\otimes u_\chi).
$$
Since $$(\pi(\tilde\psi(\chi)^{-1})\otimes1)U_\pi(\pi(\tilde\psi(\chi))\otimes1) =(\iota\otimes\Ad\psi(\chi))(U_\pi),$$ this exactly means that if $a\in\C[G_q]$ is a matrix coefficient of $\pi$, then $u_\chi a=(\Ad \psi(\chi))(a)u_\chi$. Therefore $\C[\tilde G^\tau_q]=\C[G_q]\rtimes_{\Ad\psi}\hat T_\tau$.

\smallskip

Now, the quantum group $G^\tau_q$ is the quotient of $\tilde G^\tau_q$ defined by the category of $\tau$-admissible representations. By definition, a representation $\pi\otimes\chi$ of $\Uni^\tau_q(\liealg{g})$ is $\tau$-admissible if $\pi(\tau_i)=\chi(\tau_i)$. Therefore $\C[G_q^\tau]\subset \C[\tilde G^\tau_q]=\C[G_q]\rtimes_{\Ad\psi}\hat T_\tau$ is spanned by elements of the form $au_\chi$, where $a$ is a matrix coefficient of an admissible representation $\pi$ such that $\pi(\tau_i)=\chi(\tau_i)$. If $\pi$ is irreducible, then $\pi(\tau_i)$ is scalar, and we have $\rho(\tau_i)(a)=\pi(\tau_i)a$. Hence
$
\C[G_q^\tau]=(\C[G_q]\rtimes_{\Ad\psi}\hat T_\tau)^{T_\tau}.
$
\end{proof}

\begin{cor}
The C$^*$-algebra $C(G_q^\tau)$ is of type I.
\end{cor}

\bp Since $C(G^\tau_q)\subset C(G_q) \rtimes_{\Ad \psi} \hat{T}_\tau$, this follows from the known fact that the C$^*$-algebra $C(G_q)$ is of type I.
\ep

Recall that the family $(C(G_q))_{0 < q < \infty}$ has canonical structure of a continuous field of C$^*$-al\-gebras~\cite{MR2861394}.

\begin{cor}\label{cor:cont-field}
The C$^*$-algebras $(C(G_q^\tau))_{0 < q < \infty}$ form a continuous field of C$^*$-algebras.
\end{cor}

\subsection{Primitive spectrum}
\label{sec:spectrum}

Let us turn to a description of the primitive spectrum of $C(G_q^\tau)$. We will concentrate on the case $q\ne1$, the case $q=1$ can be treated similarly. First of all observe that the action of $T_\tau$ on $C(G_q) \rtimes_{\Ad \psi} \hat{T}_\tau$ is saturated, since every spectral subspace contains a unitary. We thus obtain a strong Morita equivalence
\begin{equation} \label{eq:Morita}
C(G_q^\tau) \sim_M C(G_q) \rtimes_{\Ad \psi} \hat{T}_\tau\rtimes_{\rho,\widehat{\Ad \psi}}{T_\tau}\cong C(G_q)\rtimes_\rho T_\tau\rtimes_{\Ad\psi,\hat\rho}\hat T_\tau.
\end{equation}

Recall how to describe primitive spectra of crossed products, see e.g.~\cite{MR2288954}. Let $\Gamma$ be a finite group acting on a separable C$^*$-algebra $A$. Then any primitive ideal $J$ of $A \rtimes \Gamma$ is determined by the $\Gamma$-orbit of an ideal $I \in \Prim(A)$ and an ideal $J_0 \in \Prim(A \rtimes \Stab_\Gamma(I))$ by the condition $J_0 \cap A = I$ and $J = \Ind J_0$.

If $A$ is of type I, the ideals $J_0$ can, in turn, be described as follows. Put $\Gamma_0=\Stab_\Gamma(I)$. We want to describe irreducible representations of $A\rtimes\Gamma_0$ whose restrictions to $A$ have kernel $I$. Let $H$ be the space of an irreducible representation $\pi$ of~$A$ with kernel $I$. Then the action of $\Gamma_0$ on $A/I$ is implemented by a projective unitary representation $\gamma\mapsto u_\gamma$ of $\Gamma_0$ on $H$. Let~$\omega$ be the corresponding $2$-cocycle. Consider the regular $\bar\omega$-representation $\gamma\mapsto\lambda^{\bar\omega}_\gamma$ of $\Gamma_0$ on $\ell^2(\Gamma_0)$. Then $A\rtimes \Gamma_0$ has a representation on $H\otimes\ell^2(\Gamma_0)$ defined by $a\mapsto \pi(a)\otimes1$, $\gamma\mapsto u_\gamma\otimes\lambda^{\bar \omega}_\gamma$. Any irreducible representation of $A\rtimes \Gamma_0$ whose restriction to $A$ has kernel $I$ is a subrepresentation of this representation. So it remains to decompose the representation of $A\rtimes\Gamma_0$ on $H\otimes\ell^2(\Gamma_0)$ into irreducible subrepresentations. The von Neumann algebra generated by the image of $A\rtimes \Gamma_0$ is $B(H)\otimes C^*(\Gamma_0;\bar\omega)$. Therefore the representations we are interested in are in a one-to-one correspondence with irreducible representations of $C^*(\Gamma_0;\bar\omega)$.

To summarize, if $A$ is a separable C$^*$-algebra of type I and $\Gamma$ is a finite group acting on $A$, then the primitive spectrum $\Prim(A\rtimes\Gamma)$ can be identified with the set of pairs $([I],J)$, where $[I]$ is the $\Gamma$-orbit of an ideal $I\in\Prim(A)$, $J\in \Prim(C^*(\Gamma_I;\bar\omega_I))$, and $\omega_I$ is the $2$-cocycle on $\Gamma_I=\Stab_\Gamma(I)$ defined by a projective representation of $\Gamma_I$ implementing the action of $\Gamma_I$ on the image of $A$ under an irreducible representation with kernel $I$.

\smallskip

Returning to $C(G^\tau_q)$, for an element $w\in W$ of the Weyl group and a character $\chi\in\hat T_\tau$, put $\theta_w(\chi) = w^{-1}(\tilde{\psi}(\chi))\tilde{\psi}(\chi)^{-1}$.  This defines a homomorphism from $\hat{T}_\tau$ to $T$.

\begin{prop}
For $q>0$, $q\ne1$, the primitive spectrum of $C(G_q^\tau)$ can be identified with
\[
\coprod_{w \in W} (\theta_w(\hat{T}_\tau) \backslash T / T_\tau) \times \widehat{\theta_w^{-1}(T_\tau)}.
\]
\end{prop}

\begin{proof} In view of the strong Morita equivalence~\eqref{eq:Morita} it suffices to describe the primitive spectrum~of
$$
C(G_q)\rtimes_\rho T_\tau\rtimes_{\Ad\psi,\hat\rho}\hat T_\tau.
$$

Recall that the spectrum of $C(G_q)$ is $W \times T$. The right translation action of $T_\tau$ on $C(G_q)$ defines an action on $W\times T$ that is simply the action by translations on $T$. Therefore $\Prim(C(G_q)\rtimes_\rho T_\tau)$ can be identified with $W\times T/T_\tau$, and every irreducible representation of $C(G_q)\rtimes_\rho T_\tau$ is induced from an irreducible representation of $C(G_q)$.

Next, we have to understand the action of $\hat T_\tau$ on $\Prim(C(G_q)\rtimes_\rho T_\tau)$. Since the dual action preserves the equivalence class of any induced representation, we just have to look at the action~$\Ad\psi$. Given a representation $\pi_w\otimes\pi_t$ of $C(G_q)$ corresponding to $(w,t)\in W\times T$, we have
$$
(\pi_w\otimes\pi_t)(\Ad \psi(\chi^{-1}))\sim\pi_w\otimes\pi_{\theta_w(\chi)t}
$$
by~\citelist{\cite{MR2914062}*{Lemma~3.4}\cite{MR3009718}*{Lemma~8}}. It follows that the action of $\hat T_\tau$ on $\Prim(C(G_q)\rtimes_\rho T_\tau)=W\times T/T_\tau$ is by translations on $T/T_\tau$ via the homomorphisms $\theta_w\colon\hat T_\tau\to T$. Hence the space of $\hat T_\tau$-orbits is $\coprod_{w \in W}\theta_w(\hat{T}_\tau) \backslash T / T_\tau$, and the stabilizer of a point $(w,tT_\tau)$ is $\theta_w^{-1}(T_\tau)\subset\hat T_\tau$.

To finish the proof of the proposition it remains to show that the action $(\Ad\psi,\hat\rho)$ of $\theta_w^{-1}(T_\tau)$ on $C(G_q)\rtimes_\rho T_\tau$ can be implemented in the space of the induced representation $\Ind(\pi_w\otimes\pi_t)$ by a unitary representation of $\theta_w^{-1}(T_\tau)$. For this, in turn, it suffices to show that the equivalences
$$
(\pi_w\otimes\pi_{t'})(\Ad t^{-1})\sim\pi_w\otimes\pi_{w^{-1}(t)t^{-1}t'}
$$
from \citelist{\cite{MR2914062}*{Lemma~3.4}\cite{MR3009718}*{Lemma~8}} can be implemented by a unitary representation $t\mapsto v_t$ of $T/Z(G)$ on the space of representation $\pi_w$. But this is easy to see. Specifically, using the notation of~\cites{MR2914062,MR3009718}, if $w=s_i$ is the reflection corresponding to a simple root $\alpha_i$, then the required representation $t\mapsto v_t$ on $\ell^2(\Z_+)$ can be defined by $v_te_n=\pairing{t}{\alpha_i}^ne_n$. For arbitrary $w$ we just have to take tensor products of such representations.
\end{proof}

\begin{remk}
A description of the topology on $\Prim(C(G_q))$ is given in~\cite{MR2914062}. The above argument is, however, not quite enough to understand the topology on $\Prim(C(G^\tau_q))$.
\end{remk}

\subsection{K-theory}
\label{sec:k-theory}

The maximal torus $T$ is embedded in $\Uni(G^\tau_q)$, so it can be considered as a subgroup of $G_q^\tau$.  Let us consider the right translation action $\rho$ of $T$ on $C(G_q^\tau)$.  The crossed product $C(G_q^\tau) \rtimes_\rho T$ is a $\hat{T}$-C$^*$-algebra with respect to the dual action.

\begin{prop}\label{prop:KhatT-isom}
The dual action of $\hat{T}$ on $C(G_q^\tau) \rtimes_\rho T$ is equivariantly strongly Morita equivalent to an action on $C(G_q) \rtimes_\rho T$ that is homotopic to the dual action.
\end{prop}

\begin{proof}
If we identify $C(G_q^\tau)$ with $(C(G_q)\rtimes_{\Ad\psi}\hat T_\tau)^{T_\tau}$, then the action of $T$ by right translations on~$C(G_q^\tau)$ extends to an action on $C(G_q)\rtimes_{\Ad\psi}\hat T_\tau$ that is trivial on $C^*(T_\tau)$ and coincides with the action by right translations on $C(G_q)$. This action of $T$ on $C(G_q)\rtimes_{\Ad\psi}\hat T_\tau$ commutes with the action of $T_\tau$. Hence the strong Morita equivalence~\eqref{eq:Morita} is $T$-equivariant, and taking crossed products we get a $\hat T$-equivariant strong Morita equivalence
\begin{equation} \label{eq:Morita2}
C(G_q^\tau) \rtimes_\rho T\sim_M C(G_q)\rtimes_{\Ad\psi}\hat T_\tau\rtimes_{\rho,\widehat{Ad\psi}}{T_\tau}\rtimes_\rho T.
\end{equation}
Denote the C$^*$-algebra on the right hand side by $A$. We claim that $A$ is isomorphic to $$B = C(G_q) \rtimes_{\Ad \psi} \hat{T}_\tau \rtimes_{\widehat{\Ad \psi}} T_\tau \rtimes_\rho T.$$  Indeed, the map
$au_\chi u_tu_{t'}\mapsto au_{\chi}u_tu_{tt'}$ for $a\in C(G_q)$, $\chi\in\hat T_\tau$, $t\in T_\tau$ and $t'\in T$ is the required isomorphism. The dual action of $\hat T$ on $A$ corresponds to an action $\beta$ on $B$ which is given by the dual action on the copy of $C^*(T)$ and by the dual action on the copy of $C^*(T_\tau)$ via the canonical homomorphism $r\colon \hat{T} \rightarrow \hat{T}_\tau$.

The map $\hat T\ni\chi\mapsto u_{r(\chi)}\in C^*(\hat T_\tau)\subset M(B)$ is a $1$-cocycle for the action $\beta$. Therefore $\beta$ is strongly Morita equivalent to the action $\gamma$ defined by $\gamma_\chi = (\Ad {u_{r(\chi)}}) \beta_\chi$. This action is already trivial on~$C^*(T_\tau)$, while on $C(G_q)$  it is given by $\Ad {{\psi}(r(\chi))}$, and on $C^*(T)$ it coincides with the dual action.

Denote by $\delta$ the restriction of $\gamma$ to $C(G_q)\rtimes_\rho T\subset M(B)$. Then, similarly to~\eqref{eq:Morita2}, the actions $\delta$ and $\gamma$ are strongly Morita equivalent.

Combining the Morita equivalences that we have obtained, we conclude that the dual action of~$\hat T$ on $C(G^\tau_q)\rtimes_\rho T$ is strongly Morita equivalent to the action $\delta=(\Ad{{\psi}(r(\cdot))},\hat\rho)$ on $C(G_q)\rtimes_\rho T$. Choosing a basis in $\hat{T}=P$ and paths from $\tilde{\psi}(r(\chi))$ to the neutral element in $T$ for every basis element~$\chi$, we see that $\delta$ is homotopic to the dual action on $C(G_q)\rtimes_\rho T$.
\end{proof}

\begin{thm}\label{thm:kk-isom}
The C$^*$-algebra $C(G_q^\tau)$ is $\KK$-isomorphic to $C(G_q)$, hence to $C(G)$.
\end{thm}

\begin{proof}
Since the torsion-free commutative group $\hat{T}$ satisfies the strong Baum--Connes conjecture, the functor $A\mapsto A\rtimes\hat T$ maps homotopic actions into $\KK$-isomorphisms of the corresponding crossed products. By the previous proposition, this, together with the Takesaki--Takai duality, implies that $C(G^\tau_q)$ and $C(G_q)$ are $\KK$-isomorphic. By~\cite{MR2914062} we also know that $C(G_q)$ is $\KK$-isomorphic to~$C(G)$.
\end{proof}

\begin{remk} \mbox{\ }\newline\noindent
(i) The above proof shows that the continuous field of Corollary~\ref{cor:cont-field} is a $\KK$-fibration in the sense of~\cite{MR2511635}.  The argument of~\cite{MR2861394} applies to the Dirac operator $D$ given by Remark~\ref{remk:isospec-deform-Dirac}, and we obtain that the K-homology class of $D$ is independent of $q$.  The bi-equivariance of $D$ and the construction in the proof of Proposition~\ref{prop:KhatT-isom} imply that the K-homology class of $D$ is also independent of $\tau$ up to the isomorphism of Theorem~\ref{thm:kk-isom}.
\newline\noindent
(ii) For the group $\hat T$ the strong Baum--Connes conjecture is a consequence of the Pimsner--Voiculescu sequence in $\KK$-theory. Therefore the proof of Theorem~\ref{thm:kk-isom} can be written such that it relies only on this sequence, see e.g.~\cite{sangha}*{Section~5.1} for a related argument.
\end{remk}

\bigskip

\section{Twisted \texorpdfstring{$\SU_q(n)$}{SUq(n)}}

\subsection{Special unitary group}
\label{seq:rep-thry-SUn}

Let us review the structure of $\SU(n)$, see e.g.~\cite{MR1153249}*{Chapter~15}.  For the sake of presentation, it is convenient to consider also the unitary group $\U(n)$.  We take the subgroup of the diagonal matrices $\tilde{T}$ as a maximal torus of $\U(n)$, and take $T = \tilde{T} \cap \SU(n)$ as a maximal torus of $\SU(n)$.  We will often identify $\tilde T$ with $\T^n$. We write the corresponding Cartan subalgebras as $\tilde{\liealg{h}} \subset \liealg{gl}_n$ and $\liealg{h} \subset \liealg{sl}_n$.

Let $\{e_{i j}\}_{i, j = 1}^n$ be the matrix units in $M_n(\C) = \liealg{gl}_n$, and $\{\tilde L_i\}^n_{i=1}$ be the basis in $\tilde{\liealg{h}}^*$ dual to the basis $\{e_{ii}\}^n_{i=1}$ in $\tilde{\liealg{h}}$. Denote by $L_i$ the image of $\tilde L_i$ in ${\liealg{h}}^*$. Therefore any $n-1$ elements among $L_1,\dots, L_n$ form a basis in ${\liealg{h}}^*$, and we have $\sum_iL_i=0$.

The weight lattice $P\subset{\liealg{h}}^*$ is generated by the elements $L_i$. The pairing between $T$ and $P$ is given by $\pairing{t}{L_i}=t_i$ for $t\in T\subset\T^n$. As simple roots we take
$$
\alpha_i=L_i-L_{i+1}, \ \ 1\le i\le n-1.
$$
The fundamental weights are then given by
$$
\varpi_i = L_1 + \dots + L_i, \ \ 1\le i\le n-1.
$$

Consider the homomorphism $|\cdot|\colon P\to\Z$ such that $L_1\mapsto n-1$ and $L_i\mapsto-1$ for $1<i\le n$. In other words,
$$
|a_1\varpi_1+\dots+a_{n-1}\varpi_{n-1}| =\lambda_1+\dots+\lambda_{n-1},
$$
where $\lambda_{n-i}$ is given by $a_1 + \cdots + a_i$. The image of $Q$ under $|\cdot|$  is $n\Z$, and therefore we can use this homomorphism to identify $P/Q$ with $\Z/n\Z$.

\subsection{Twisted quantum special unitary groups}
\label{sec:quant-spec-unit}

By Proposition~\ref{p3cocycles}, the cohomology group $H^3(\Z/n\Z;\T)$ is isomorphic to $\Z/n\Z$, and a cocycle generating this group can be defined by
$$
\phi(a,b,c)=\zeta_n^{\omega_n(a,b)c},\ \ \text{where}\ \ \zeta_n=e^{2\pi i/n}\ \ \text{and}\ \ \omega_n(a,b)=\biggl\lfloor\frac{a+b}{n}\biggr\rfloor
-\biggl\lfloor\frac{a}{n}\biggr\rfloor - \biggl\lfloor\frac{b}{n}\biggr\rfloor.
$$
Using this generator we identify $H^3(\Z/n\Z;\T)$ with the group $\mu_n\subset\T$ of units of order $n$. Therefore, given $\zeta\in\mu_n$, we have a category $\Rep(\SU_q(n))^\zeta$ with associativity morphisms defined by multiplication by $\zeta^{\omega_n(|\lambda|,|\eta|)|\nu|}$ on the tensor product $V_\lambda\otimes V_\eta\otimes V_\nu$ of irreducible $\Uni_q(\liealg{g})$-modules with highest weights $\lambda,\eta,\nu$. This agrees with the conventions of Kazhdan and Wenzl~\cite{MR1237835}.

It is also convenient to identify $Z(\SU(n))$ with the group $\mu_n$. Thus, for $\tau=(\tau_1,\dots,\tau_{n-1})\in\mu_n^{n-1}$, we can define a twisting $\SU^\tau_q(n)$ of $\SU_q(n)$. Its representation category is one of $\Rep(\SU_q(n))^\zeta$, and to find $\zeta$ we have to compute the homomorphism $\Theta\colon Z(\SU(n))^{n-1}\to H^3(P/Q;\T)$ introduced in Section~\ref{sec:isomorphisms}. Under our identifications this becomes a homomorphism $\mu_n^{n-1}\to\mu_n$.

\begin{prop}
We have $\Theta(\tau)=\prod^{n-1}_{i=1}\tau_i^{-i}$.
\end{prop}

\bp Recall the construction of $\Theta$. We choose a function $f\colon P\times P\to\T$ such that it factors through $P\times (P/Q)$ and
$
f(\lambda+\alpha_i,\mu)=\overline{\pairing{\tau_i}{\mu}}f(\lambda,\mu).
$
Then $\Theta(\tau)$ is the cohomology class of $\partial f$ in $H^3(P/Q;\T)$.

Note that $\pairing{\tau_i}{\mu}=\tau_i^{-|\mu|}$, which is immediate for $\mu=L_j$, and define a character $\chi$ of $Q\otimes(P/Q)=Q\otimes(\Z/n\Z)$ by
$$
\chi(\alpha_i\otimes k)=\tau_i^k\ \ \text{for}\ \ 1\le i\le n-1\ \ \text{and}\ \ k\in\Z/n\Z,
$$
so that $f(\lambda+\alpha,\mu)=\chi(\alpha\otimes|\mu|)f(\lambda,\mu)$ for all $\alpha\in Q$. By Proposition~\ref{pTor}, the cohomology class of $\partial f$ depends only on the restriction of $\chi$ to
$$
\ker(Q\otimes(\Z/n\Z)\to P\otimes(\Z/n\Z))\cong\Tor^\Z_1(\Z/n\Z,\Z/n\Z)\cong\Z/n\Z,
$$
and by varying $\tau$ we get this way an isomorphism $\Hom(\Tor^\Z_1(\Z/n\Z,\Z/n\Z),\T)\cong H^3(\Z/n\Z;\T)$. In order to compute this isomorphism we can use the resolution $n\Z\to\Z\to\Z/n\Z$ instead of $Q\to P\xrightarrow{|\cdot|}\Z/n\Z$. Define a morphism between these resolutions by $\Z\to P$, $1\mapsto\varpi_{n-1}=-L_n$. By pulling back $\chi$ under this morphism, we get a character $\tilde\chi$ of $(n\Z)\otimes(\Z/n\Z)$ such that
$$
\tilde\chi(n\otimes k)=\chi(n\varpi_{n-1}\otimes k).
$$
We have $n\varpi_{n-1}=\sum^{n-1}_{i=1}i\alpha_i$. Therefore
$$
\tilde\chi(n\otimes k)=\zeta^k,\ \ \text{where}\ \ \zeta=\prod^{n-1}_{i=1}\tau_i^i.
$$
Then the function $\tilde f\colon\Z\times\Z\to\T$ defined by
$$
\tilde f(a,b)=\zeta^{\left\lfloor\frac{a}{n}\right\rfloor b},
$$
factors through $\Z\times(\Z/n\Z)$, $\tilde f(a+n,b)=\tilde\chi(n\otimes b)\tilde f(a,b)$ and $(\partial\tilde f)(a,b,c)=\zeta^{-\omega_n(a,b)c}$. Therefore the class of $\partial\tilde f$ in $H^3(\Z/n\Z;\T)=\mu_n$ is $\zeta^{-1}$.
\ep

In Section~\ref{sec:isomorphisms} we also introduced a homomorphism $\Upsilon$. In the present case we have $H^2(P/Q;\T)=0$, so $\Upsilon$ is a homomorphism $\ker\Theta\to H^2(P;\T)$.

\begin{lem}
The homomorphism $\Upsilon\colon \ker\Theta\to H^2(P;\T)$ is injective.
\end{lem}

\bp Assume $\tau\in\ker\Theta$, so $\prod^{n-1}_{i=1}\tau_i^i=1$. In this case the character $\chi$ of $Q\otimes (P/Q)$ from the proof of the previous proposition extends to $P\otimes(P/Q)$ by
$$
\chi(L_i\otimes\mu)=(\tau_1\dots\tau_{i-1})^{-|\mu|}\ \ \text{for}\ \ 1\le i\le n\ \ \text{and}\ \ \mu\in P.
$$
Therefore if we consider $\chi$ as a function on $P\times P$, we can take it as a function $f$ in that proof. Then~$f$ is a $2$-cocycle, and by definition, the image of~$\tau$ under $\Upsilon$ is the cohomology class of $\bar f$. It is well-known, and also follows from Proposition~\ref{pabelian}, that $f$ is a coboundary if and only if $f$ is symmetric. For $1<i<j\le n$ we have
$$
f(L_i,L_j)\overline{f(L_j,L_i)}=(\tau_i\dots\tau_{j-1})^{-1}.
$$
So if $f$ is symmetric, then $\tau_2=\dots=\tau_{n-1}=1$, but then also $\tau_1=1$.
\ep

Therefore Proposition~\ref{piso} does not give us any nontrivial isomorphisms between the quantum groups $\SU^\tau_q(n)$. On the other hand, the flip map on the Dynkin diagram induces an automorphism of $\Uni(\SU_q(n))$ such that $K_i\mapsto K_{n-i}$ and $E_i\mapsto E_{n-i}$ for $1\le i\le n-1$. On $Z(\SU(n))\subset \Uni(\SU_q(n))$ this automorphism is $t\mapsto t^{-1}$. It follows that it induces isomorphisms
$$
\SU^{(\tau_1,\dots,\tau_{n-1})}_q(n)\cong \SU^{(\tau^{-1}_{n-1},\dots,\tau_1^{-1})}_q(n).
$$
For $0<q<1$, these seem to be the only obvious isomorphisms between the quantum groups $\SU^\tau_q(n)$.

\subsection{Generators and relations}
\label{sec:generators-relations}

The C$^*$-algebra $C(\SU_q(n))$ is generated by the matrix coefficients $(u_{i j})_{1 \le i, j \le n}$ of the natural representation of $\SU_q(n)$ on $\C^n$, the fundamental representation with highest weight~$\varpi_1$.  They  satisfy the relations~\citelist{\cite{MR934283}\cite{MR943923}}
\begin{gather}
\label{eq:SUqN-rel-1}
  u_{i j} u_{i l} = q u_{i l} u_{i j} \quad (j < l), \quad u_{i j} u_{k j} = q u_{k j} u_{i j} \quad (i < k),\\
\label{eq:SUqN-rel-2}
  u_{i j} u_{k l} = u_{k l} u_{i j} \quad (i > k, j < l), \quad u_{i j} u_{k l} - u_{k l} u_{i j} = (q - q^{-1}) u_{i l} u_{k j} \quad (i < k, j < l),\\
\label{eq:SUqN-rel-3}
 \qdet((u_{i j})_{i, j}) =  \sum_{\sigma \in S_n} (-q)^{\absv{\sigma}} u_{1 \sigma(1)} \cdots u_{n \sigma(n)} = 1.
\end{gather}
Here, $\absv{\sigma}$ is the inversion number of the permutation $\sigma$. The involution is defined by $$u_{ij}^*=(-q)^{j-i}\qdet(U^{\hat i}_{\hat j}),$$ where $U^{\hat i}_{\hat j}$ is the matrix obtained from $U=(u_{kl})_{k,l}$ by deleting the $i$-th row and $j$-th column.

In order to find generators and relations of $\C[\SU^\tau_q(n)]$, we will use the embedding of the algebra $\C[\SU^\tau_q(n)]$ into $\C[\SU_q(n)]\rtimes_{\Ad\psi}\hat T_\tau$ described in Theorem~\ref{thm:fixed-pt-in-crossed-prod}. Recall that $\psi\colon\hat T_\tau\to T/Z(\SU(n))=T/\mu_n$ is the homomorphism such that $\pairing{\tilde\psi(\chi)}{\alpha_i}=\chi(\tau_i)$, where $\tilde\psi(\chi)$ is a lift of $\psi(\chi)$ to $T$. Hence
$$
\tilde\psi(\chi)=(z,z\chi(\tau_1)^{-1},\dots, z\chi(\tau_1\dots\tau_{n-1})^{-1})\in T\subset \T^n,
$$
where $z\in\T$ is a number such that $z^n=\prod^{n-1}_{i=1}\chi(\tau_i)^{-i}$. It follows that
\begin{equation} \label{eaction}
(\Ad\psi(\chi))(u_{ij})=\Bigl(\prod_{1\le p<i}\chi(\tau_p)\Bigr)
\Bigl(\prod_{1\le p<j}\chi(\tau_p)^{-1}\Bigr)u_{ij}.
\end{equation}

Now, the algebra $\C[\SU^\tau_q(n)]$ is generated by matrix coefficients of the fundamental representation of $\SU^\tau_q(n)$ with highest weight $\varpi_1$. Under the embedding $\C[\SU^\tau_q(n)]\hookrightarrow\C[\SU_q(n)]\rtimes_{\Ad\psi}\hat T_\tau$, these matrix coefficients correspond to $v_{ij}=u_{ij}u_{\chi_\nat}$, where $\chi_\nat\in\hat T_\tau$ is the character determined by the natural representation of $\SU_q(n)$ on $\C^n$, so $\chi_\nat(\tau_i)=\tau_i$. From \eqref{eq:SUqN-rel-1}-\eqref{eq:SUqN-rel-3} we then get the following relations:
\begin{gather}
\label{eq:SUqtN-rel-1}
  v_{i j} v_{i l} = \Bigl ( \prod_{j \le p < l} \tau_p^{-1} \Bigr ) q v_{i l} v_{i j} \quad (j < l), \quad  v_{i j} v_{k j} = \Bigl ( \prod_{i \le p < k} \tau_p \Bigr ) q v_{k j} v_{i j} \quad (i < k),\\
\label{eq:SUqtN-rel-2}
  v_{i j} v_{k l} = \Bigl ( \prod_{i < p \le k} \tau_p^{-1} \Bigr ) \Bigl ( \prod_{j \le p < l} \tau_p^{-1} \Bigr ) v_{k l} v_{i j} \quad (i > k, j < l),\\
\label{eq:SUqtN-rel-3}
  \Bigl ( \prod_{j \le p < l} \tau_p \Bigr ) v_{i j} v_{k l} - \Bigl ( \prod_{i \le p < k} \tau_p \Bigr ) v_{k l} v_{i j} = (q - q^{-1}) v_{i l} v_{k j} \quad (i < k, j < l),\\
\label{eq:SUqtN-rel-4}
  \sum_{\sigma \in S_n} \tau^{m(\sigma)} (-q)^{\absv{\sigma}} v_{1 \sigma(1)} \cdots v_{n \sigma(n)} = 1,
\end{gather}
where $m(\sigma)=(m(\sigma)_1,\dots,m(\sigma)_{n-1})$ is the multi-index given by $m(\sigma)_i = \sum_{k=2}^n (k - 1) m^{(k, \sigma(k))}_i$, and
$$
m^{(k,j)}_i=\begin{cases}1,&\text{if}\ \ k\le i<j,\\
-1,&\text{if}\ \ j\le i<k,\\
0,&\text{otherwise}.\end{cases}
$$

\begin{prop}
For any $\tau\in\mu_n^{n-1}$, the algebra $\C[\SU^\tau_q(n)]$ is a universal algebra generated by elements $v_{ij}$ satisfying relations \eqref{eq:SUqtN-rel-1}-\eqref{eq:SUqtN-rel-4}.
\end{prop}

\bp We already know that relations \eqref{eq:SUqtN-rel-1}-\eqref{eq:SUqtN-rel-4} are satisfied in $\C[\SU^\tau_q(n)]$, so we just have to show that there are no other relations. Let $\A$ be a universal algebra generated by elements $w_{ij}$ satisfying relations \eqref{eq:SUqtN-rel-1}-\eqref{eq:SUqtN-rel-4}. We can define an action of $\hat T_\tau$ on $\A$ by~\eqref{eaction}. Then in $\A\rtimes\hat T_\tau$ the elements $w_{ij}u_{\chi_\nat}^{-1}$ satisfy the defining relations of $\C[\SU_q(n)]$, so we have a homomorphism $\C[\SU_q(n)]\to \A\rtimes\hat T_\tau$ mapping $u_{ij}$ into $w_{ij}u_{\chi_\nat}^{-1}$. It extends to a homomorphism $\C[\SU_q(n)]\rtimes\hat T_\tau\to \A\rtimes\hat T_\tau$ that is identity on the group algebra of $\hat T_\tau$. Restricting to $\C[\SU^\tau_q(n)]\subset \C[\SU_q(n)]\rtimes\hat T_\tau$, we get a homomorphism $\C[\SU^\tau_q(n)]\to\A$ mapping $v_{ij}$ into $w_{ij}$.
\ep

The involution on $\C[\SU^\tau_q(n)]$ is determined by requiring the invertible matrix $(v_{ij})_{i,j}$ to be unitary. An explicit formula can be easily found using that for $\C[\SU_q(n)]$.

\begin{remk}
The relations in $\C[\SU^\tau_q(n)]$ cannot be obtained using the FRT-approach, since the categories $\Rep(\SU_q(n))^\zeta$ are typically not braided. More precisely, $\Rep(\SU_q(n))^\zeta$ has a braiding  if and only if either $\zeta=1$ or $n$ is even and $\zeta=-1$. This statement is already implicit in~\cite{MR1237835}, and it can be proved as follows. If $\zeta=1$ or $n$ is even and $\zeta=-1$, then a braiding indeed exists, see e.g.~\cite{MR2307417}. Conversely, suppose we have a braiding. In other words, there exists an $R$-matrix $\RR$ for $(\Uni(\SU_q(n)),\Dhat_q,\Phi)$, where $\Phi=\zeta^{\omega_n(|\lambda|,|\eta|)|\nu|}$. Recall that this means that $\RR$ is an invertible element in $\Uni(\SU_q(n)\times\SU_q(n))$ such that
$\Dhat^{\rm op}_q =\RR\Dhat_q(\cdot )\RR^{-1}$ and
\begin{equation*}
\label{equasi1}
(\Dhat_q\otimes\iota)(\RR)=\Phi_{312}\RR_{13}\Phi_{132}^{-1}
\RR_{23}\Phi, \ \
(\iota\otimes\Dhat_q)(\RR)=\Phi_{231}^{-1}\RR_{13}\Phi_{213}
\RR_{12}\Phi^{-1}.
\end{equation*} 
Since $\Phi$ is central and symmetric in the first two variables, the last two identities can be written as
$$
(\Dhat_q^{\rm op}\otimes\iota)(\RR)=\RR_{23}\RR_{13}\Phi, \ \
(\iota\otimes\Dhat_q)(\RR)=\RR_{13}\RR_{12}\Phi_{321}^{-1}.
$$
On the other hand, we know that $\Rep(\SU_q(n))$ is braided, so there exists an element $\RR_q$ satisfying the above properties with $\Phi$ replaced by $1$. Consider the element $F=\RR_q^{-1}\RR$. Then $F$ is invariant, meaning that it commutes with the image of $\Dhat_q$. Furthermore, we have
\begin{align*}
(F\otimes1)(\Dhat_q\otimes\iota)(F)
&=(\RR_q^{-1}\otimes1)(\Dhat_q^{\rm op}\otimes\iota)(\RR_q^{-1})(\Dhat_q^{\rm op}\otimes\iota)(\RR)(\RR\otimes1)\\
&=((\RR_q)_{23}(\RR_q)_{13}(\RR_q)_{12})^{-1}\RR_{23}\RR_{13}\RR_{12}\Phi,
\end{align*}
and similarly
\begin{align*}
(1\otimes F)(\iota\otimes\Dhat_q)(F)&=(\iota\otimes\Dhat_q)(\RR_q^{-1}) (1\otimes\RR_q^{-1}) (1\otimes\RR)(\iota\otimes\Dhat_q)(\RR)\\
&=((\RR_q)_{23}(\RR_q)_{13}(\RR_q)_{12})^{-1}\RR_{23}\RR_{13}\RR_{12}\Phi_{321}^{-1}.
\end{align*}
Therefore
$$
(\iota \otimes \hat{\Delta}_q)(F^{-1}) (1 \otimes F^{-1}) (F \otimes 1) (\hat{\Delta}_q \otimes \iota)(F)=\Phi_{321}\Phi.
$$
This implies that $\Rep(\SU_q(n))$ is monoidally equivalent to $\Rep(\SU_q(n))^{\Phi_{321}\Phi}$. Since the cocycle $\Phi_{321}\Phi$ on the dual of the center is cohomologous to the cocycle $\zeta^{2\omega_n(|\lambda|,|\eta|)|\nu|}$, this means that $\Rep(\SU_q(n))$ is monoidally equivalent to $\Rep(\SU_q(n))^{\zeta^2}$. By the Kazhdan--Wenzl classification this is the case only if $\zeta^2=1$.
\end{remk}

\bigskip

\appendix

\section{Cocycles on abelian groups}
\label{sec:cocycl-abel-groups}

Let $\Gamma$ be a discrete abelian group.  As is common in operator algebra, we denote the generators of the group algebra $\Z[\Gamma]$ by $\lambda_\gamma$ ($\gamma\in\Gamma$). Let $(C_*(\Gamma),d)$ be the nonnormalized bar-resolution of the $\Z[\Gamma]$-module $\Z$, so $C_n(\Gamma)$ ($n\ge0$) is the free $\Z[\Gamma]$-module with basis consisting of $n$-tuples of elements in $\Gamma$, written as $[\gamma_1|\dots|\gamma_n]$, and the differential $d\colon C_n(\Gamma)\to C_{n-1}(\Gamma)$ is defined by
$$
d[\gamma_1|\dots|\gamma_n]=\lambda_{\gamma_1}[\gamma_2|\dots|\gamma_n]+\sum^{n-1}_{i=1}(-1)^{i}
[\gamma_1|\dots|\gamma_i+\gamma_{i+1}|\dots|\gamma_n]+(-1)^n[\gamma_1|\dots|\gamma_{n-1}].
$$

Let $M$ be a commutative group endowed with the trivial $\Gamma$-module structure.  The group cohomology $H^*(\Gamma; M)$ can be computed from the standard complex induced by the bar-resolution.  Concretely, we have a cochain complex
\[
C^*(\Gamma; M) = \Hom_{\Z[\Gamma]}(C_*(\Gamma), M) = \Map(\Gamma^*, M),
\]
endowed with the boundary map $\partial\colon C^n(\Gamma;M) \rightarrow C^{n + 1}(\Gamma;M)$ defined by
\begin{multline*}\label{eq:grp-chm-std-cplx}
(\partial\phi)(\gamma_1, \ldots, \gamma_{n + 1}) = \phi(\gamma_2, \ldots, \gamma_{n + 1}) - \phi(\gamma_1 + \gamma_2, \gamma_3, \ldots, \gamma_{n + 1}) +\dots\\
+ (-1)^n \phi(\gamma_1, \ldots, \gamma_{n - 1}, \gamma_n +\gamma_{n + 1}) + (-1)^{n + 1} \phi(\gamma_1, \ldots, \gamma_n).
\end{multline*}
By $M$-valued cocycles on $\Gamma$ we mean cocycles in $(C^*(\Gamma;M),\partial)$. We will consider only $\T$-valued cocycles, but with minor modifications everything what we say remains true for cocycles with values in any divisible group $M$.

For the sake of computation, it is also convenient to introduce the integer homology $H_*(\Gamma)=H_*(\Gamma;\Z)$, which is given as the homology of the complex $C_*(\Gamma;\Z)=\Z\otimes_{\Z[\Gamma]}C_*(\Gamma)$.  Since the action of $\Gamma$ on $\T$ is trivial, we have $C^*(\Gamma;\T)=\Hom_{\Z[\Gamma]}(C_*(\Gamma),\T)=\Hom(C_*(\Gamma;\Z),\T)$.  Moreover, the injectivity of $\T$ as a $\Z$-module implies that any character of $H_n(\Gamma; \Z)$ can be lifted to a character of~$C_n(\Gamma;\Z)$.  It follows that the groups $H^n(\Gamma; \T)$ and $H_n(\Gamma)$ are Pontryagin dual to each other. This is a particular case of the Universal Coefficient Theorem.

A map $\phi\colon\Gamma^n\to\T$ ($n\ge1$) is called an $n$-character on $\Gamma$ if it is a character in every variable, so it is defined by a character on $\Gamma^{\otimes n}$ (unless specified otherwise, all tensor products in this appendix are over $\Z$). It is easy to see that every $n$-character is a $\T$-valued cocycle. An $n$-character $\phi$ is called alternating if $\phi(\gamma_1,\dots,\gamma_n)=1$ as long as $\gamma_i=\gamma_{i+1}$ for some $i$; then $\phi(\gamma_{\sigma(1)},\dots,\gamma_{\sigma(n)})=\phi(\gamma_1,\dots,\gamma_n)^{\sgn(\sigma)}$ for any $\sigma\in S_n$. In other words, a $n$-character is alternating if it factors through the exterior power $\bigwedge^n\Gamma$, which is the quotient of $\Gamma^{\otimes n}$ by the subgroup generated by elements $\gamma_1\otimes\dots\otimes\gamma_n$ such that $\gamma_i=\gamma_{i+1}$ for some $i$. It will sometimes be convenient to view $\bigwedge^n\Gamma$ as a subgroup of $\Gamma^{\otimes n}$ via the embedding
$$
\gamma_1\wedge\dots\wedge\gamma_n\mapsto\sum_{\sigma\in S_n}\sgn(\sigma)\gamma_{\sigma(1)}\otimes\dots\otimes\gamma_{\sigma(n)}.
$$

We will also consider $\bigwedge^n\Gamma$ as a subgroup of $H_n(\Gamma)$. The embedding $\bigwedge^*\Gamma\hookrightarrow H_*(\Gamma)$ is constructed using the canonical isomorphism $\Gamma\cong H_1(\Gamma)$ and the Pontryagin product on $H_*(\Gamma)$, see \cite{MR1324339}*{Theorem~V.6.4}. On the chain level the latter product can be defined using the shuffle product, so that $\gamma_1\wedge\dots\wedge\gamma_n$ is identified with the homology class of the cycle
$$
\sum_{\sigma\in S_n}\sgn(\sigma)(1\otimes[\gamma_{\sigma(1)}|\dots|\gamma_{\sigma(n)}])\in C_n(\Gamma;\Z).
$$

For free abelian groups we have $\bigwedge^*\Gamma=H_*(\Gamma)$. By duality we get the following description of cocycles.

\begin{prop} \label{pabelian}
If $\Gamma$ is free abelian, then for every $n\ge 1$ we have:
\enu{i} any $\T$-valued $n$-cocycle on $\Gamma$ is cohomologous to an alternating $n$-character;
\enu{ii} an $n$-character is a coboundary if and only if it vanishes on $\bigwedge^n\Gamma\subset\Gamma^{\otimes n}$; in particular, an alternating $n$-character is a coboundary if and only its order divides~$n!$.
\end{prop}

\bp The value of an $n$-cocycle $\phi$ on $\gamma_1\wedge\dots\wedge\gamma_n\in H_n(\Gamma)$ is
$$
\pairing{\phi}{\gamma_1\wedge\dots\wedge\gamma_n}=\prod_{\sigma\in S_n}\phi(\gamma_{\sigma(1)},\dots,\gamma_{\sigma(n)})^{\sgn(\sigma)}.
$$
This immediately implies (ii), since if $\phi$ is an $n$-character, then the above product is exactly the value of $\phi$ on $\gamma_1\wedge\dots\wedge\gamma_n$ considered as an element of $\Gamma^{\otimes n}$.

Turning to (i), assume $\psi$ is an $n$-cocycle. It defines a character $\chi$ of $H_n(\Gamma)=\bigwedge^n\Gamma$. Let $\phi$ be a character of $\bigwedge^n\Gamma$ such that $\phi^{n!}=\chi$. Then $\phi$ is an alternating $n$-character, and $\phi$ is cohomologous to~$\psi$, since both cocycles $\phi$ and $\psi$ define the same character $\chi$ of $H_n(\Gamma)=\bigwedge^n\Gamma$.
\ep

We now turn to the more complicated case of finite abelian groups and concentrate on $3$-cocycles. In this case $\bigwedge^3\Gamma$ is a proper subgroup of $H_3(\Gamma)$: as follows from Proposition~\ref{p3cocycles} below, the quotient $H_3(\Gamma)/\bigwedge^3\Gamma$ is (noncanonically) isomorphic to $\Gamma\oplus(\Gamma\bigwedge\Gamma)$. Correspondingly, not every third cohomology class can be represented by a $3$-character. Additional $3$-cocycles can be obtained by the following construction.

\begin{lem} \label{lconstruction}
Assume $\Gamma=\Gamma_1/\Gamma_0$ for some abelian groups $\Gamma_1$ and $\Gamma_0$. Suppose $f\colon\Gamma_1\times\Gamma_1\to\T$ is a function such that
$$
f(\alpha,\beta+\gamma)=f(\alpha,\beta)\ \ \text{and}\ \ f(\alpha+\gamma,\beta)=\chi(\gamma\otimes\beta)f(\alpha,\beta)
$$
for all $\alpha,\beta\in\Gamma_1$ and $\gamma\in\Gamma_0$, where $\chi$ is a character of $\Gamma_0\otimes\Gamma$. Then the function
$$
(\partial f)(\alpha,\beta,\gamma)=f(\beta,\gamma)f(\alpha+\beta,\gamma)^{-1}f(\alpha,\beta+\gamma)f(\alpha,\beta)^{-1}
$$
on $\Gamma_1^3$ is $\Gamma^3_0$-invariant,  hence it defines a $\T$-valued $3$-cocycle on $\Gamma$.
\end{lem}

\bp This is a straightforward computation.
\ep

In order to describe explicitly generators of $H^3(\Gamma;\T)$, let us introduce some notation. For natural numbers $n_1,\dots,n_k$, denote by $(n_1,\dots,n_k)$ their greatest common divisor. For $n\in\N$, denote by~$\chi_n$ the character of $\Z/n\Z$ defined by $\chi_n(1)=e^{2\pi i/n}$. Finally, for integers~$a$ and~$b$ and a natural number~$n$, put
$$
\omega_n(a,b)=\biggl\lfloor\frac{a+b}{n}\biggr\rfloor
-\biggl\lfloor\frac{a}{n}\biggr\rfloor-\biggl\lfloor\frac{b}{n}\biggr\rfloor.
$$
Note that $\omega_n$ is a well-defined function on $\Z/n\Z\times\Z/n\Z$ with values $0$ or $1$.

\begin{prop} \label{p3cocycles}
Assume $\Gamma=\oplus^m_{i=1}\Z/n_i\Z$ for some $n_i\ge1$. Then
$$
H^3(\Gamma;\T)\cong\bigoplus_{i}\Z/n_i\Z\oplus\bigoplus_{i<j}\Z/(n_i,n_j)\Z\oplus\bigoplus_{i<j<k}\Z/(n_i,n_j,n_k)\Z.
$$
Explicitly, generators $\phi_i$ of $\Z/n_i\Z$, $\phi_{ij}$ of $\Z/(n_i,n_j)\Z$ and $\phi_{ijk}$ of $\Z/(n_i,n_j,n_k)\Z$ can be defined by
\begin{align*}
\phi_i(a,b,c)&=\chi_{n_i}(\omega_{n_i}(a_i,b_i)c_i),&
\phi_{ij}(a,b,c)&=\chi_{n_j}(\omega_{n_i}(a_i,b_i)c_j),&
\phi_{ijk}(a,b,c)&=\chi_{(n_i,n_j,n_k)}(a_ib_jc_k).
\end{align*}
\end{prop}

\bp Recall first how to compute the homology of finite cyclic groups. Consider the group $\Z/n\Z$. Then there is a free resolution $(P_*,d)$ of the $\Z[\Z/n\Z]$-module~$\Z$ such that $P_k$ is generated by one basis element $e_k$, and
$$
de_{2k+1}=\lambda_1 e_{2k}-e_{2k}\ \ \text{and}\ \ de_{2k+2}=\sum_{a\in\Z/n\Z}\lambda_ae_{2k+1}\ \ \text{for}\ \ k\ge0.
$$
The morphism $P_0\to\Z$ is given by $e_0\mapsto1$. Using this resolution we get
$$
H_{2k+1}(\Z/n\Z)\cong\Z/n\Z\ \ \text{and}\ \ H_{2k+2}(\Z/n\Z)=0 \ \ \text{for}\ \ k\ge0.
$$

\smallskip

Turning to the proof of the proposition, the first statement is equivalent to
$$
H_3(\Gamma)\cong\bigoplus_{i}\Z/n_i\Z\oplus\bigoplus_{i<j}\Z/(n_i,n_j)\Z\oplus\bigoplus_{i<j<k}\Z/(n_i,n_j,n_k)\Z.
$$
This, in turn, is proved by induction on $m$ using the isomorphisms
$$
H_1(\Gamma)\cong\Gamma, \ \ H_2(\Gamma)\cong\Gamma\bigwedge\Gamma,
$$
which are valid for any abelian group $\Gamma$, and the K\"unneth formula, which gives that $H_3(\Gamma\oplus\Z/n\Z)$ is isomorphic to
$$
H_3(\Gamma)\oplus (H_2(\Gamma)\otimes H_1(\Z/n\Z))\oplus H_3(\Z/n\Z)\oplus\Tor^\Z_1(H_1(\Gamma),H_1(\Z/n\Z)).
$$
Note only that
$$
\Tor^\Z_1(\Z/k\Z,\Z/n\Z)\cong\Z/(k,n)\Z\cong\Z/k\Z\otimes\Z/n\Z.
$$

\smallskip

Let us check next that the functions $\phi_i$, $\phi_{ij}$ and $\phi_{ijk}$ are indeed $3$-cocycles. For $\phi_{ijk}$ this is clear, since it is a $3$-character. Concerning $\phi_{i}$, consider the function
$$
f_i(a,b)=\chi_{n_i}\left(-\left\lfloor\frac{a_i}{n_i}\right\rfloor b_i\right)
$$
on $\Z^m\times\Z^m$. It is of the type described in Lemma~\ref{lconstruction} for $\Gamma_1=\Z^m$ and $\Gamma_0=\oplus^m_{i=1}n_i\Z$, so
$\phi_i(a,b,c)=(\partial f_i)(a,b,c)$ is a $3$-cocycle on $\Gamma$. Similarly, consider the function
$$
f_{ij}(a,b)=\chi_{n_j}\left(-\left\lfloor\frac{a_i}{n_i}\right\rfloor b_j\right).
$$
It is again of the type described in Lemma~\ref{lconstruction}, so $\phi_{ij}=\partial f_{ij}$ is a $3$-cocycle.

\smallskip

Our next goal is to construct a `dual basis' in $H_3(\Gamma)$. Let $u_i$ be the generator $1\in\Z/n_i\Z\subset\Gamma$. Denote by $\theta_{ijk}$ the cycle representing $u_i\wedge u_j\wedge u_k\in\bigwedge^3\Gamma\subset H_3(\Gamma)$ obtained by the  shuffle product, so
$$
\theta_{ijk}=\sum_{\sigma\in S_3}\sgn(\sigma)(1\otimes[u_{\sigma(i)}|u_{\sigma(j)}|u_{\sigma(k)}]),
$$
where we consider $S_3$ as the group of permutations of $\{i,j,k\}$.

Consider the $\Z[\Z/n_i\Z]$-resolution $(P^i_*,d)$ of $\Z$ described at the beginning of the proof. Let $e^i_n$ be the basis element of $P^i_n$. We have a chain map $P^i_*\to C_*(\Z/n_i\Z)$ of resolutions of $\Z$ defined~by
\begin{equation}\label{echain}
e^i_0\mapsto[\emptyset],\ \ e^i_1\mapsto[1],\ \ e^i_2\mapsto\sum_{a\in\Z/n_i\Z}[a|1],\ \ e^i_3\mapsto\sum_{a\in\Z/n_i\Z}[1|a|1],....
\end{equation}
It follows that we have a $3$-cycle $\theta_i\in C_3(\Gamma;\Z)$ defined by
$$
\theta_i=\sum^{n_i-1}_{a=0}1\otimes[u_i|au_i|u_i].
$$

Finally, consider the $\Z[\Z/n_i\Z\oplus\Z/n_j\Z]$-resolution $P_*^i\otimes P_*^j$ of $\Z$. Using this resolution we get a third homology class represented by
$$
\frac{n_j}{(n_i,n_j)}1\otimes e^i_2\otimes e^j_1+\frac{n_i}{(n_i,n_j)}1\otimes e^i_1\otimes e^j_2.
$$
A chain map between the resolutions $P^i_*\otimes P^j_*$ and $C_*(\Z/n_i\Z\oplus\Z/n_j\Z)$ can be defined by the tensor product of the chain maps \eqref{echain} and the shuffle product. This gives us a $3$-cycle $\theta_{ij}\in C_3(\Gamma;\Z)$. Explicitly,
\begin{multline*}
\theta_{ij}=\frac{n_j}{(n_i,n_j)}\sum^{n_i-1}_{a=0}1\otimes([au_i|u_i|u_j]-[au_i|u_j|u_i]+[u_j|au_i|u_i])\\
+\frac{n_i}{(n_i,n_j)}\sum^{n_j-1}_{b=0}1\otimes([u_i|bu_j|u_j]-[bu_j|u_i|u_j]+[bu_j|u_j|u_i]).
\end{multline*}

The only nontrivial pairings between the cocycles $\phi_i$, $\phi_{ij}$, $\phi_{ijk}$ and the cycles $\theta_i$, $\theta_{ij}$, $\theta_{ijk}$ are
$$
\pairing{\phi_i}{\theta_i}=\zeta_{n_i}, \ \ \pairing{\phi_{ij}}{\theta_{ij}}=\zeta_{n_j}^{{n_j}/{(n_i,n_j)}}=\zeta_{(n_i,n_j)},\ \ \pairing{\phi_{ijk}}{\theta_{ijk}}=\zeta_{(n_i,n_j,n_k)},
$$
where $\zeta_n=e^{2\pi i/n}$. This implies that these cocycles and cycles are the required generators of the Pontryagin dual groups $H^3(\Gamma;\T)$ and $H_3(\Gamma)$.
\ep

\begin{cor}\label{ccmain}
Assume $\Gamma$ is a finite abelian group. Write $\Gamma$ as $\Gamma_1/\Gamma_0$ for a finite rank free abelian group $\Gamma_1$. Then for any $\T$-valued $3$-cocycle $\phi$ on $\Gamma$ the following conditions are equivalent:
\enu{i} $\phi$ vanishes on $\bigwedge^3\Gamma\subset H_3(\Gamma)$;
\enu{ii} $\phi$ lifts to a coboundary on $\Gamma_1$;
\enu{iii} $\phi=\partial f$ for a function $f\colon\Gamma_1\times\Gamma_1\to\T$ as in Lemma~\ref{lconstruction}.
\end{cor}

\bp The equivalence of (i) and (ii) is clear, since a cocycle on $\Gamma_1$ is a coboundary if and only if it vanishes on $H_3(\Gamma_1)=\bigwedge^3\Gamma_1$. Also, obviously (iii) implies (ii). Therefore the only nontrivial statement is that (i), or (ii), implies (iii). Assume $\phi$ is a cocycle that vanishes on $\bigwedge^3\Gamma\subset H_3(\Gamma)$. We can identify $\Gamma_1$ with $\Z^m$ in such a way that $\Gamma_0=\oplus^{m}_{i=1}n_i\Z$ for some $n_i\ge1$. Then in the notation of the proof of the above proposition the assumption on $\phi$ means that $\phi$ vanishes on the cycles $\theta_{ijk}$, whose homology classes are exactly $u_i\wedge u_j\wedge u_k\in \bigwedge^3\Gamma\subset H_3(\Gamma)$. It follows that~$\phi$ is cohomologous to product of powers of cocycles $\phi_i$ and $\phi_{ij}$. But the cocycles $\phi_i$ and $\phi_{ij}$ are of the form $\partial f$ with $f\colon\Gamma_1\times\Gamma_1\to\T$ as in Lemma~\ref{lconstruction}. Therefore $\phi$ is cohomologous to a cocycle of the form $\partial f$, hence $\phi$ itself is of the same form.
\ep

Since every character of $\bigwedge^3\Gamma\subset\Gamma^{\otimes 3}$ extends to a $3$-character on $\Gamma$, this corollary can also be formulated as follows.

\begin{cor} \label{ccmain2}
With $\Gamma=\Gamma_1/\Gamma_0$ as in the previous corollary, any $\T$-valued $3$-cocycle~$\phi$ on $\Gamma$ can be written as product of a $3$-character $\chi$ on $\Gamma$ and a cocycle $\partial f$ with $f\colon\Gamma_1\times\Gamma_1\to\T$ as in Lemma~\ref{lconstruction}. Such a cocycle $\phi$ lifts to a coboundary on $\Gamma_1$ if and only if $\chi$ vanishes on $\bigwedge^3\Gamma\subset\Gamma^{\otimes 3}$, and in this case $\phi=\partial g$ with $g\colon\Gamma_1\times\Gamma_1\to\T$ as in Lemma~\ref{lconstruction}.
\end{cor}

Let us now look more carefully at the construction of cocycles described in Lemma~\ref{lconstruction}. As Corollary~\ref{ccmain} shows, the class of $3$-cocycles obtained by this construction does not depend on the presentation of $\Gamma$ as quotient of a finite rank free abelian group. It is also clear that there is a lot of redundancy in this construction, since the group $H_3(\Gamma)$ can be much smaller than $\Gamma_0\otimes\Gamma$. The following proposition makes these observations a bit more precise.

\begin{prop} \label{pTor}
Assume $\Gamma$ is a finite abelian group, and write $\Gamma$ as $\Gamma_1/\Gamma_0$ for a finite rank free abelian group~$\Gamma_1$. Let $f\colon\Gamma_1\times\Gamma_1\to\T$ be a function as in Lemma~\ref{lconstruction}, and $\chi$ be the associated character of~$\Gamma_0\otimes\Gamma$ . Then the cohomology class of $\partial f$ in $H^3(\Gamma;\T)$ depends only on the restriction of $\chi$ to
$$
\ker(\Gamma_0\otimes\Gamma\to\Gamma_1\otimes\Gamma)\cong\Tor^\Z_1(\Gamma,\Gamma)\cong \Gamma\otimes\Gamma.
$$
Therefore by varying $\chi$ we get a natural in $\Gamma$ homomorphism
$$
\Hom(\Tor^\Z_1(\Gamma,\Gamma),\T)\to H^3(\Gamma;\T),
$$
whose image is the annihilator of $\bigwedge^3\Gamma\subset H_3(\Gamma)$.
\end{prop}

\bp It is easy to see that the cohomology class of $\partial f$ depends only on $\chi$, so we have a homomorphism
$
\Hom(\Gamma_0\otimes\Gamma,\T)\to H^3(\Gamma;\T).
$
We have to check that if a character $\chi$ of $\Gamma_0\otimes\Gamma$ vanishes on $\ker(\Gamma_0\otimes\Gamma\to\Gamma_1\otimes\Gamma)$, then the image of~$\chi$ in $H^3(\Gamma;\T)$ is zero. But this is clear, since we can extend~$\chi$ to a character $f$ of $\Gamma_1\otimes\Gamma$, and then $f$, considered as a function on $\Gamma_1\times\Gamma_1$, is of the type described in Lemma~\ref{lconstruction}, with associated character $\chi$, and $f$ is a $2$-character, so $\partial f=0$.

Naturality of the homomorphism $\Hom(\Tor^\Z_1(\Gamma,\Gamma),\T)\to H^3(\Gamma;\T)$ in $\Gamma$ is straightforward to check. The statement that its image coincides with the annihilator of $\bigwedge^3\Gamma\subset H_3(\Gamma)$ follows from Corollary~\ref{ccmain}.
\ep

\raggedright


\bigskip

\end{document}